\documentclass[11pt]{amsart}

\makeatletter
\g@addto@macro{\endabstract}{\@setabstract}
\newcommand{\authorfootnotes}{\renewcommand\thefootnote{\@fnsymbol\c@footnote}}%
\makeatother
\usepackage[margin=1.3in]{geometry}
\usepackage{setspace}
\usepackage{lipsum}
\usepackage{amsfonts}
\usepackage{graphicx}
\usepackage{epstopdf}
\usepackage{algorithmic}
\usepackage{mathrsfs}
\usepackage{enumerate}
\usepackage{todonotes}
\usepackage{amsmath}

\usepackage[nocompress]{cite}

\usepackage{amssymb}

\usepackage{epsfig}  		
\usepackage{epic,eepic}       

 \usepackage[foot]{amsaddr}

\newtheorem{theorem}{Theorem}[section]
\newtheorem{proposition}[theorem]{Proposition}
\newtheorem{lemma}[theorem]{Lemma}

\theoremstyle{definition}

\theoremstyle{remark}

\numberwithin{equation}{section}

\newcommand{\R}{\mathbf{R}}  

\def\R{{\mathbb{R}}}

\def\dxi{{{\mathrm d}\xi}}
\def\dt{{{\mathrm d} t}}

\def\dx{{{\mathrm d} x}}
\def\dpp{{{\mathrm d} p}}

\newcommand{\cW}{\mathcal{W}}
\newcommand{\cWe}{\cW_\varepsilon}
\newcommand{\sWe}{\mathscr{W}_\varepsilon}
\newcommand{\cG}{\mathcal{G}}
\newcommand{\sG}{\mathscr{G}}
\newcommand{\wxi}{\widetilde{\xi}}
\newcommand{\Dt}{{\delta t}}
\newcommand{\rme}{\mathrm{e}}
\newcommand{\Id}{I} 
\renewcommand{\wp}{\widetilde{p}}
\allowdisplaybreaks
\title{Friction-adaptive descent: a family of dynamics-based optimization methods}

\begin{document}
\maketitle
\begin{center}
  \normalsize
  \authorfootnotes
  Aikaterini Karoni\footnote{katerinakaron@gmail.com}\textsuperscript{1}, Benedict Leimkuhler\footnote{b.leimkuhler@ed.ac.uk}\textsuperscript{1},
  Gabriel Stoltz\footnote{gabriel.stoltz@enpc.fr}\textsuperscript{2} \par \bigskip

  \textsuperscript{1}School of Mathematics, University of Edinburgh, Edinburgh EH9 2NX, Scotland \par
  \textsuperscript{2}CERMICS, Ecole des Ponts, Marne-la-Vall\'ee, France \& MATHERIALS team-project, Inria Paris, France\par \bigskip

\end{center}



%
%
%

\begin{abstract}
We describe a family of descent algorithms which generalizes common existing schemes used in applications such as  neural network training and more broadly for optimization of smooth functions--potentially for global optimization, or as a local optimization method to be deployed within global optimization schemes.  By introducing an auxiliary degree of freedom we create a dynamical system with improved stability, reducing oscillatory modes and accelerating convergence to minima. The resulting algorithms are simple to implement, and convergence can be shown directly by Lyapunov's second method.

Although this framework, which we refer to as friction-adaptive descent (FAD), is fairly general, we focus most of our attention on a specific variant: kinetic energy stabilization (which can be viewed as a zero-temperature Nos\'e--Hoover scheme with added dissipation in both physical and auxiliary variables), termed KFAD (kinetic FAD).  To illustrate the flexibility of the FAD framework we consider several other methods. In certain asymptotic limits, these methods  can be viewed as introducing cubic damping in various forms; they can be more efficient than linearly dissipated Hamiltonian dynamics (LDHD).

We present details of the numerical methods and show convergence for both the continuous and discretized dynamics in the convex setting by constructing Lyapunov functions. The methods are tested using a toy model (the Rosenbrock function).  We also demonstrate the methods for structural optimization for atomic clusters in Lennard--Jones and Morse potentials. The experiments show the relative efficiency and robustness of FAD in comparison to LDHD.

\end{abstract}



\section{Introduction}

\label{sec:intro}
We study the design of extended systems for exploring energy (or loss) landscapes in high dimensions.  We suppose that we are given some objective function $f:\R^d \rightarrow \R$, differentiable and bounded below, and the goal is to find a local (or possibly global) minimum of $f$. To facilitate analytical study, we will assume that the function $f$ is convex, so that the minimizer is unique.  Ultimately we expect such local minimization strategies to be deployed as part of more complicated global optimization procedures or as part of stochastic optimization methods.   

There is a wide body of literature available on optimization methods (see e.g. \cite{BoVa2004,NoWr2006}), but we focus on the narrower class of iterative methods which are constructed by discretization of continuous dynamical systems.  In this work we assume that the gradient of $f$ is tractable, even if costly to evaluate. A standard iterative scheme to find local minima of~$f$ problem is gradient descent \cite{nesterov2003introductory}, which generates the sequence of iterates $x_1, x_2, \ldots$, beginning from a given initial condition~$x_0$, via 
\[
x_{n+1} = x_n - \delta t \nabla f(x_n).
\]
This method can be viewed as the Euler discretization of the continuous dynamics~$\dot{x} = -\nabla f(x)$. Gradient descent is an effective optimization strategy in many situations. Combined with the concept of stochastic gradients from Bayesian data analysis, it is widely used in deep learning (see for instance~\cite[Chapter~8]{Goodfellow-et-al-2016}).  However, it has limitations including slow convergence and instability (for large values of $\delta t$).   One way to increase the flexibility of iterative methods is to add a vector of momenta ($p\in \R^d$), that is to treat the extended system with state described by $(x,p)$ in which successive iterates are generated from
\[
(x_{n+1}, p_{n+1}) = T_{\delta t} (x_n,p_n),
\]
where $T_{\delta t}$ represents an approximation of the time $\delta t$ flow map of some suitable system.  For example $T_{\delta t}$ might be obtained as a numerical discretization of
\begin{align}
\frac{\dx}{\dt}  & = p, \label{eq:lindiss-1}\\
\frac{\dpp}{\dt}  & = -\nabla f(x) - \gamma p. \label{eq:lindiss-2}
\end{align}
Here $\gamma >0$ is a parameter which ensures that $p\rightarrow 0$ over time and which ultimately guarantees the approach to a local minimum of $f$.  This idea of extending the system is very natural by reference to Newtonian mechanics.  Define the energy $H(x,p) = \|p\|^2/2 + f(x)$ (where $\|p\|^2/2$ represents the kinetic energy and $f(x)$ the potential energy), then (\ref{eq:lindiss-1})-(\ref{eq:lindiss-2}) represents a physical system subject to linearly dissipative forces, i.e. mechanical damping.    Upon introducing finite differences into (\ref{eq:lindiss-1})-(\ref{eq:lindiss-2}) one arrives at various discrete procedures.  In the optimization community, this is typically referred to as Polyak's heavy ball method \cite{Po1964}, although we here use the term linearly dissipated Hamiltonian descent (LDHD). A variant of this approach, corresponding to a particular discretization of (\ref{eq:lindiss-1})-(\ref{eq:lindiss-2}), is termed Nesterov accelerated gradient~\cite{Nesterov}.

Since the goal of extending the gradient dynamics~$\dot{x}=-\nabla f(x)$ by introducing new variables is to solve the optimization problem (or obtain new candidate optimization schemes), there is nothing that anchors us to physically motivated dynamics.  It is  natural to further generalize the framework in various ways.  The only essential criteria  are the following:
\begin{enumerate}[1.]
\item correctness -- the extended system should have accessible mathematical properties and allow for a thorough convergence analysis; :
\item efficiency -- the extended system should be easy to implement and computationally tractable in the sense of adding minimal overhead in the computation of a timestep compared to, say, gradient descent, and not less robust in practice;  
\item value -- at the end of the day, there should be some tangible benefit from the added complexity of the extension.
\end{enumerate}

A family of optimization methods of this type was recently studied in \cite{maddison2018hamiltonian}, including the use of more general kinetic energies~$k(p)$ and models for dissipation~$D(p)$, resulting from discretization of
\begin{align}
\frac{\dx}{\dt}  & = \nabla k(p), \label{gen-1}\\
\frac{\dpp}{\dt}  & = -\nabla f(x) - D(p). \label{gen-2}
\end{align}
One of the key outputs of the study of~\eqref{gen-1}-\eqref{gen-2} is that, by carefully choosing~$k$ as a symmetrization of the convex conjugate of~$f$ while keeping the usual friction~$D(p) = p$, a geometric convergence rate can be obtained for strongly convex and smooth functions~$f$, with a convergence rate which does not depend on the upper/lower bounds on the Hessian. The convergence result can in fact be extended to non-smooth or non-strongly convex functions, and to discretizations of~\eqref{gen-1}-\eqref{gen-2}.  Motivation for general kinetic energies comes from works like~\cite{StTr2018} which have shown potential for acceleration in the sampling context.

The reasons for using the straightforward generalization (\ref{gen-1})-(\ref{gen-2}) are related to the criteria 1-3.  It is easy to construct Lyapunov functions for (\ref{gen-1})-(\ref{gen-2}), numerical implementation can be performed easily, and there are some clear wins--examples where this framework proves beneficial. However, the choice of $D(p)$ is tied to~$D(p)=\gamma p$ in~\cite{maddison2018hamiltonian} in order to guarantee a convergence rate independent of the strong convexity and smoothness parameters of~$f$ (see in particular Remark~1 there).   The linear dissipation case corresponds to the class of models known as ``conformal Hamiltonian systems" which implies a geometric principle underpinning the flow of the system \cite{McPe2001,BhFlMo2016}, although this property is not much exploited in optimization treatments. All numerical experiments in \cite{maddison2018hamiltonian} focused on $D(p) = \gamma p$.

Another way of generalizing LDHD is to consider the damping coefficient as a variable parameter.  An example of this is the method of Su, Boyd and Candes \cite{Su2016}, which generalizes Nesterov descent by replacing $\gamma$ with $\gamma(t) \propto t^{-1}$.  A recent work of Moucer, Taylor and Bach \cite{mo2022} demonstrated convergence for such methods using Lyapunov functions.  There have also been a few approaches introduced in the 
machine learning literature (see, e.g., \cite{Su2021}) to  determine what they term the `momentum parameter', akin to  adjusting $\gamma$ during the descent process.  From our perspective, these schemes are  heuristic and sacrifice the intuitive foundation of the Polyak model. 

In this article, we consider generalizations that go beyond (\ref{gen-1})-(\ref{gen-2}) by introducing additional variables along with nonlinear negative feedback loop control laws.    This framework is suggested in work of Maxwell \cite{Maxwell} and Wiener \cite{Wiener}, but the methods studied here are more directly related to the Nos\'{e}--Hoover sampling methods \cite{NH1,NH2,NH3} that are widely used in molecular modelling.  These are `thermostats' that regulate the kinetic energy of system.   In the case of a standard Hamiltonian system, the Nos\'{e}--Hoover thermostat introduces an adaptive friction coefficient
\begin{align}
\frac{\dx}{\dt}  & = p, \label{nh-1}\\
\frac{\dpp}{\dt}  & = -\nabla f(x) - \xi p, \label{nh-2}\\
\frac{\dxi}{\dt}  & = \|p\|^2 - \beta^{-1}, \label{nh-3}
\end{align}
where $\beta$ is a parameter (which can be viewed as a ``reciprocal temperature'').   In contrast to some other frequently encountered extended systems (e.g. generalized Langevin equation \cite{GLE}) the coupling between auxiliary variables and physical ones in (\ref{nh-1})-(\ref{nh-3}) is {\em nonlinear} which potentially complicates both numerics and analysis.

It can be shown that (\ref{nh-1})-(\ref{nh-3}) preserves the probability measure with density 
\[
\rho(x,p,\xi) = Z^{-1} \exp \left (-\beta \left [\|p\|^2/2 + f(x) + \xi^2/2 \right ] \right ),
\]
where $Z$ is a normalization constant so that $\int_{\R^{2d+1}} \rho(x,p,\xi) \dx \, \dpp \, \dxi =1$. This probability measure can be interpreted as an extended canonical ensemble in statistical physics.
In case the system is ergodic, the invariant probability measure is unique and the law of the process converges to it in the long time limit.  The configurational marginal of $\rho$ is proportional to $\exp(-\beta f(x))$, thus the extended system allows sampling of a target canonical distribution.  It is intuitively apparent that when $\beta$ is large (``low temperature'') the system will be driven to the vicinity of a local minimum.  One may also see the direct control of the kinetic energy as a potentially desirable stabilizing feature during optimization.  

When (\ref{nh-3}) is augmented by linear dissipation and stochastic forcing, the resulting Nos\'{e}--Hoover--Langevin system \cite{LeNoTh2008} also can be shown to converge to the extended canonical state. The thermostatting idea can be generalized in a variety of ways, for example incorporating coordinate-dependent projections in the coupling between auxiliary variables and physical ones \cite{JiaLe2005}, an idea that we exploit later in this paper.
It can also be shown that adding noise and linear dissipation in the physical variables ($x,p$) to the Nos\'e--Hoover thermostat maintains
the canonical state \cite{JoLe2011}.   Despite the nonlinearity, numerical integration of the Nos\'{e}-Hoover system and its generalizations can be performed using straightforward splitting methods.    Moreover, theoretical analyses are now available for various forms of Nos\'{e} dynamics (see, for example, \cite{LeSaSt2020} for a hypercoercivity analysis of the Adaptive Langevin method, or~\cite{herzog2018exponential} for a Lyapunov analysis).  Drawing on these works, we explore models that incorporate thermostat-type control laws which can help to restrict or guide the approach to minima.   In contrast to existing works on thermostatting, we include linear dissipation in both physical and auxiliary variables.  

While there are many ways to embed a given optimization problem in a higher dimensional dynamical framework, our aim here is to keep the iteration steps as simple as possible, involving only explicit calculations and simplifying convergence analysis.  This leads us to consider methods which can be viewed as simple modifications of Hamiltonian dynamics:
\begin{align*}
\dot{x} & = p,\\
\dot{p} & = -\nabla f(x) - \xi \Phi(x,p) - \gamma p,\\
\dot{\xi} & = p\cdot \Phi(x,p).
\end{align*}
Computing the orbital derivative of the extended Hamiltonian 
\[
\tilde{H}(x,p,\xi) = \|p\|^2/2 + f(x) + \xi^2/2
\]
with respect to this system results in 
\[
\frac{d}{dt}\left[ \tilde{H}(x,p,\xi)\right] =  - \gamma \|p\|^2 -\xi p \cdot \Phi +\xi p \cdot \Phi = -\gamma \|p\|^2.
\]
Thus, if $\gamma >0$, 
\[
\frac{d}{dt}\left[ \tilde{H}(x,p,\xi)\right] \leq 0,
\]
with equality only when $p=0$.  The function $\tilde{H}$ is thus a weak Lyapunov function. If $\Phi$ is such that $\Phi(x,0)=0$, then any state of the form~$(x,0,\xi)$ with $\nabla f(x) = 0$ is a critical point of the dynamical system. 

The convergence here appears to rely on the linear dissipation in the physical momentum variable.   However, in many cases the method will be convergent even with $\gamma =0$.   To see this, consider the choice $\Phi(x,p) = p$, so with $\gamma=0$ we have standard Nos\'{e}--Hoover dynamics at zero temperature.   Assuming $\xi(0)>0$, we have that $\xi$ is monotone increasing and remains positive, since  $\dot{\xi} = \|p\|^2 \geq 0$.   The monotonicity is undesirable (it adds ``stiffness'', for example) and motivates the addition of friction in the equation for the auxiliary variable.  The addition of linear dissipation in $p$ ($\gamma >0$) is not always essential, but it helps to regularize the method in case of a degeneracy in the control law, as we shall see later; it also makes possible the Lyapunov stability analysis we use here to show exponential convergence of both the dynamics and discretization.

As a numerical illustration, we show the result of applying the kinetic energy-based method (termed KFAD) in the case of a system with harmonic potential $f(x) = x^TCx/2$ where $C = {\rm diag}(1,10)$. (We defer presenting  the details of the numerical methods in order to focus on the dynamical aspects; numerics will be addressed in Section~\ref{sec:numerical_schemes}. For these graphs, we initialized the system at  $(x_1(0),x_2(0))=(1,2)$ and have used a timestep $\delta t=0.01$ which was small enough to ensure that the figures are illustrative of the properties of the dynamics, not the numerical method, per se.)   A friction value of $\gamma =1$ was used in both methods.

\begin{figure}[htbp!]
  \centering
  \includegraphics[scale=0.4]{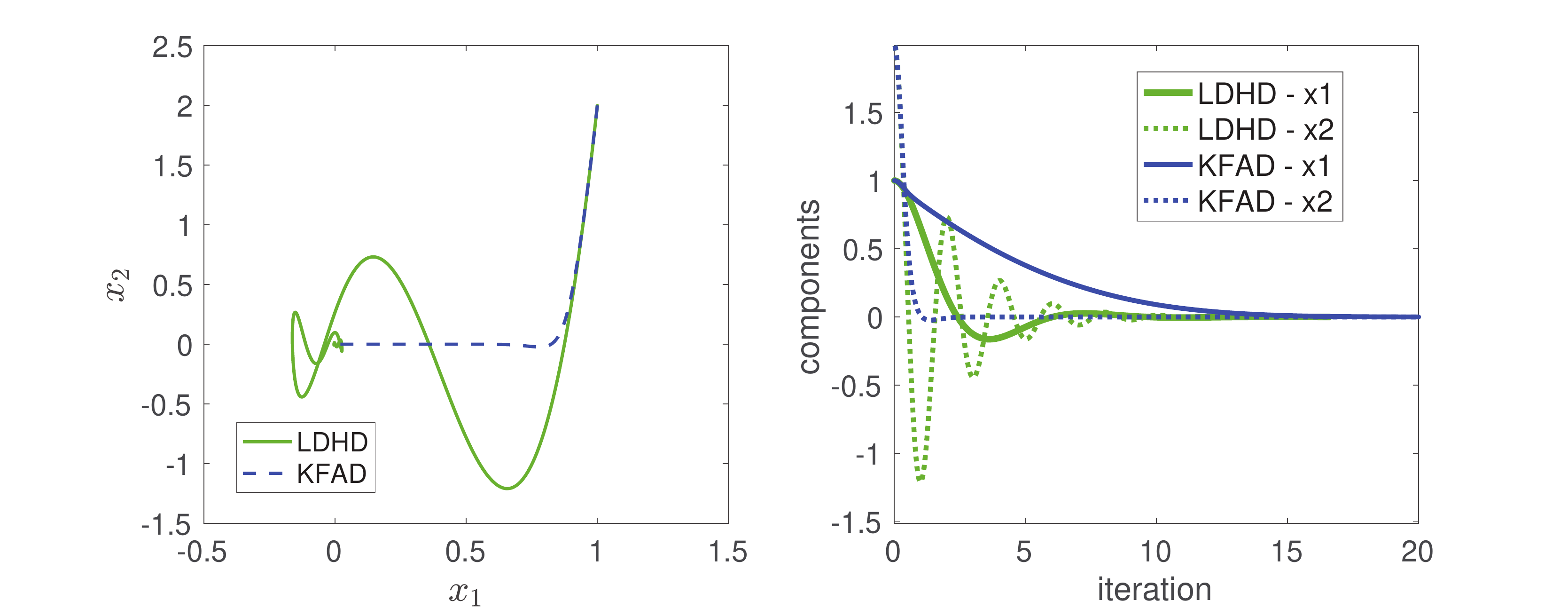} 
  \caption{The convergence of the trajectories of the linearly dissipated Hamiltonian dynamics (green) compared with that of KFAD (blue).  On the right, the graphs of the  individual components are shown as functions of time.}
  \label{fig:harmonic2d}
\end{figure}
As we can see from Fig.~\ref{fig:harmonic2d}, the extended system has  largely removed the oscillatory ``pre-equilibrium'' dynamics, in particular the strongest part of the oscillation which is due to the second component.   In these tests the stopping criterion was a simple tolerance test based on the distance to the (known) minimum $\|x-x_{\rm min}\|\leq 10^{-4}$.  The KFAD scheme used here took around 25\% more timesteps to reach the minimum (2085 vs. 1663), but as we explain later, the situation is often reversed in nonlinear systems, and more powerful methods designed using thermostats with position-dependent coupling can help to damp strong oscillations in stiff optimization problems while also providing excellent efficiency gains.

The rest of this article is organised as follows.  In Section 2 we present a family of extended systems, motivate various choices in its definition, and discuss its basic properties.  Section 3 studies convergence of the dynamical system through construction of Lyapunov functions.   Section 4 provides numerical discretizations, and the convergence of discretization schemes is taken up in Section 5.  Finally, we present numerical experiments with several variants of the FAD scheme in Section 6.
\section{Friction-adaptive descent (FAD)}
We restrict to methods that have dynamical equations that are easy to solve by common discretization methods such as splitting.  A useful class is defined by
\begin{equation} \label{A_coupling}
\Phi(x,p) = A(x)p,
\end{equation}
for some symmetric matrix function $A =A(x)$.    We also include linear damping in the equation for $\xi$  and we introduce a coefficient $\mu$ to control the coupling between momenta and auxiliary variable. Denoting by~$F(x) = -\nabla f(x)$, the equations become
\begin{align}
\dot{x} & = p,\label{eq:FADa} \\
\dot{p} & = F(x) - \xi A(x)p - \gamma p, \label{eq:FADb}\\
\dot{\xi} & = \mu^{-1}p^TA(x)p - \alpha \xi. \label{eq:FADc}
\end{align}

The inclusion of friction in the equation for $\xi$ is reminiscent of Nos\'{e}--Hoover--Langevin~\cite{LeNoTh2008}, but we do not currently propose to incorporate additive noise (or consider gradient noise here), thus (\ref{eq:FADa})-(\ref{eq:FADc}) is a deterministic system of ODEs. Without the dissipation term $-\alpha \xi$, the control law for $\xi$ would be $\dot{\xi} = \mu^{-1} p^T A(x) p$. Assuming a symmetric positive semidefinite matrix $A(x)$, we have $\dot{\xi}(t) \geq 0$ for all~$t \geq 0$, leading to a non-decreasing adaptive friction $\xi$. Adding the dissipation term allows the friction to both dynamically increase and decrease. A dissipation term also adds a ``short-term memory'' to the coupling between $\xi$ and $p$, similarly to how the dissipation term $-\gamma p$ in \eqref{eq:lindiss-2} introduces momentum (``memory'' of recent states) in Hamiltonian dynamics, leading to the dynamics of momentum gradient descent. This notion of taking history into account when designing control laws was introduced in the work of Maxwell \cite{Maxwell} on governors, where he observed that in order to achieve speed regulation with zero steady state error, the integral (history) of the speed error must be minimized. Finally, adding the dissipation term $-\alpha \xi$, ensures rigorous convergence to the equilibrium. Without this term there would be an infinite number of equilibrium points, each with a different value of $\xi$.

We are interested in methods defined by various choices of $A$.  We describe some examples below.  There are certainly many more options than these but we leave their study for future work.

\subsection{Kinetic friction-adaptive descent (KFAD)}
In case $A=I$ we have the Nos\'{e}--Hoover dynamics at zero temperature, supplemented by linear damping in both~$p$ and~$\xi$.  The distinguishing feature here is that the FAD coupling matrix $A=I$ is itself non-degenerate.

If $\alpha>0$, we can see that asymptotically, near equilibrium, $\xi \sim \alpha^{-1}\mu^{-1}p^T A p = \alpha^{-1}\mu^{-1}\|p\|^2$. Inserting this into (\ref{eq:FADb}), we have
\[
\dot{p} = F(x) - \xi p - \gamma p \sim F(x) - \frac{1}{\alpha\mu} \|p\|^2 p - \gamma p.
\]
This can be viewed as a combination of cubic damping and linear damping and is described by a generalized (quartic)
Rayleigh dissipation function \cite{Go1980,Virga2015}:
\[
R(p) = \frac{1}{4 \alpha \mu} \|p\|^4  + \frac{\gamma}{2} \|p\|^2.
\]
In the case of KFAD, the nonlinear damping is isotropic, depending only on $\|p\|$.
Cubic damping is commonly used as a vibration suppression mechanism in  mechanical systems~\cite{panananda2012effect,zhu2022} and was studied extensively in the scalar case by A. Babister \cite{Ba1974}.  We stress though that we use the freedom of the auxiliary variable to create a more
flexible dynamics. The friction parameter~$\alpha$ cannot be assumed to be large, so the relaxation to the equilibrium of the $\xi$ equation is likely to be relevant to the overall behavior; as we shall see in the numerical experiments we generally found modest values of $\alpha$ ($\alpha \approx 1$) to be most useful in practice.

\subsection{Force-based friction-adaptive descent (FFAD)}
A further generalization is to allow $A$ in (\ref{A_coupling}) to be indefinite.  For example one could use $A = F(x)F(x)^T$ where $F = -\nabla f$ is the force vector itself.  

Assuming $\alpha$ to be sufficiently large,  the auxiliary variable will relax rapidly to its equilibrium obtained by setting the right hand side
of (\ref{eq:FADc}) to zero, that is
\[
\xi \approx  \frac{1}{\alpha \mu} (p\cdot F(x))^2.
\]
The system (\ref{eq:FADa})-(\ref{eq:FADb}) then becomes
\begin{align}
    \dot{x} & =  p, \label{eq:FADa1}\\
\dot{p} & = F(x) -  \frac{1}{\alpha \mu} (p\cdot F(x))^2 F(x) F(x)^T p - \gamma p. \label{eq:FADb1}
\end{align}
We can derive the damping forces from the framework of generalized Rayleigh dissipation~\cite{Virga2015}. Specifically, define
\begin{align*}
R(x,p) & =  \frac{1}{4\alpha \mu}(p\cdot F(x))^4 + \frac{\gamma}{2} \|p\|^2\\
& = \frac{1}{4\alpha \mu} \left(p^T F(x) F(x)^T p\right)^2 + \frac{\gamma}{2} \|p\|^2.
\end{align*}
Then
\[
\nabla_p R(x,p) = \frac{1}{\alpha \mu} \left(p^T F(x) F(x)^T p\right) F(x) F(x)^T p + \gamma p,
\]
so the damping forces in (\ref{eq:FADb1}) are of the required form, i.e. $-\nabla_p R(x,p)$.
It follows that, with $H(x,p) = \|p\|^2/2 + f(x)$ the Hamiltonian of the system,
\begin{align*}
\frac{\rm d}{{\rm d}t} H & = -p \cdot \nabla_p R(x,p)\\
& = -\frac{1}{\alpha \mu} (p^T F(x) F(x)^T p)^2 - \gamma \|p\|^2.
\end{align*}

Note that (\ref{eq:FADb1}) can be rewritten as 
\begin{equation*}
    \dot{p}  = F_{\rm eff}(x)  - \gamma p,
\end{equation*}
where the ``effective'' force is given by 
\begin{equation}
    \label{eq:F_eff}
F_{\rm eff}(x) = \left(1 -  \frac{1}{\alpha \mu} (p\cdot F(x))^3\right) F(x).
\end{equation}
When $p \cdot F(x) >0 $, i.e. when we are moving in the direction of $F(x)$, for example when we are descending towards a minimum, then $F_{\rm eff}(x) = a F(x)$ with $a <1$. When $p \cdot F(x) < 0$, then $F_{\rm eff}(x) = a F(x)$ with $a > 1$, so in the case where we are moving in a direction opposite to the force, the effective force experienced by the particle increases, pulling the particle back in the direction of the force. This works as an oscillation suppression mechanism. We see that the same mechanism that suppresses vibrations also slows down descent towards the minima. Examining the dynamics \eqref{eq:FADa1}-\eqref{eq:FADb1} can therefore help us explain the behaviour of FFAD.  

The cubic damping in (\ref{eq:FADb1}) acts to rapidly diminish
$p \cdot F$, that is, the component of the momentum in the steepest descent direction.  
To demonstrate this in practice, we investigate the FFAD method for the Rosenbrock function given by
\begin{equation*} \label{eq:Rosenbrock_Potential}
    f_{R}(x_1,x_2) = (a-x_1)^2 + b(x_2-x_1^2)^2, \qquad a = 1, \ b = 100.
\end{equation*}
This function has a unique global minimum $f(x^*) = 0$ at $x^* = (1,1)$. 
The Hessian matrix of this potential has both steep and moderate components, thus we could call it a stiff optimization problem.  It is known that it presents challenges for most standard optimization schemes.

We show below the results for a typical experiment comparing KFAD and FFAD with LDHD. For KFAD/FFAD we used coefficients $\mu=1$ and $\alpha=0.1$.  For linear dissipation, present in all three methods, we used $\gamma =1$.  We started the integrator at $(1,2)$, just above the minimum. The stepsize was $\delta t =0.01$. For all methods we used splitting schemes in which various parts of the ODEs are integrated exactly; all details of these methods are presented in  Sec. 4.

\begin{figure}[htbp!]
\begin{center}
  \includegraphics[width=5.5in]{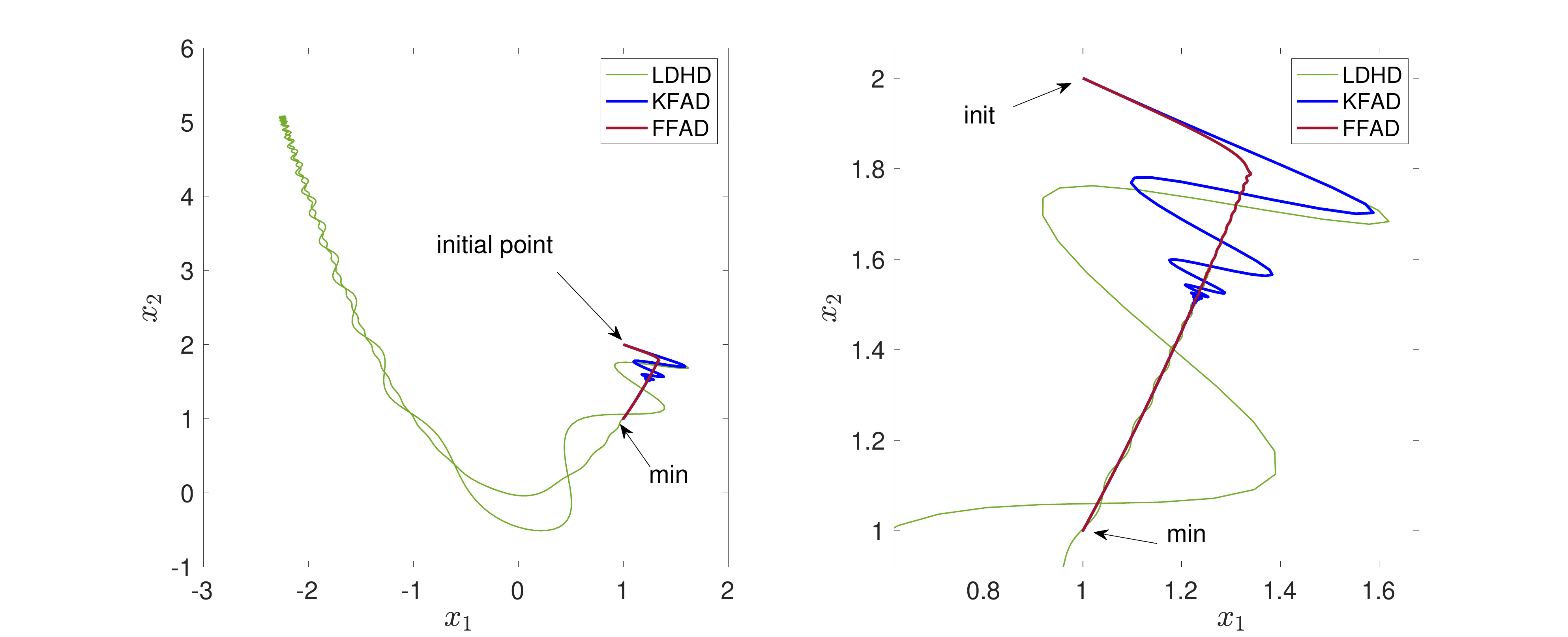}
  \caption{Comparison of performance of LDHD, FFAD and KFAD on the 2d Rosenbrock model, for initial condition $(1,2)$.  The second graph shows a close-up of motion in the vicinity of the minimum.}
  \label{FFADvKFAD}
  \end{center}
  \end{figure}

In this example, LDHD makes a rather unhelpful excursion into an extreme end of the valley (as far as the vicinity of the point $(-2,5)$), before gradually wobbling back toward the minimum.  KFAD moves rapidly into the narrow channel, and then, steadily approaches the minimum.  
The number of timesteps needed to reach the minimum (using the same criterion as before, $\|x-x_{\rm min}\| \leq 10^{-4}$) were as follows for the three methods: LDHD,1803; KFAD,1119; FFAD: 1447. 

We re-ran the example with several different initial conditions.  The results were qualitatively similar. Although the number of steps to reach the minimum varied considerably, FFAD removed nearly all oscillation outside the channel.  KFAD on the other hand was typically as or more efficient than FFAD, and reduced but did not eliminate oscillation outside the channel.    The result from a different initial condition, $(4,2)$, is shown below.

\begin{figure}[htbp!]
\begin{center}
  \includegraphics[width=5.5in]{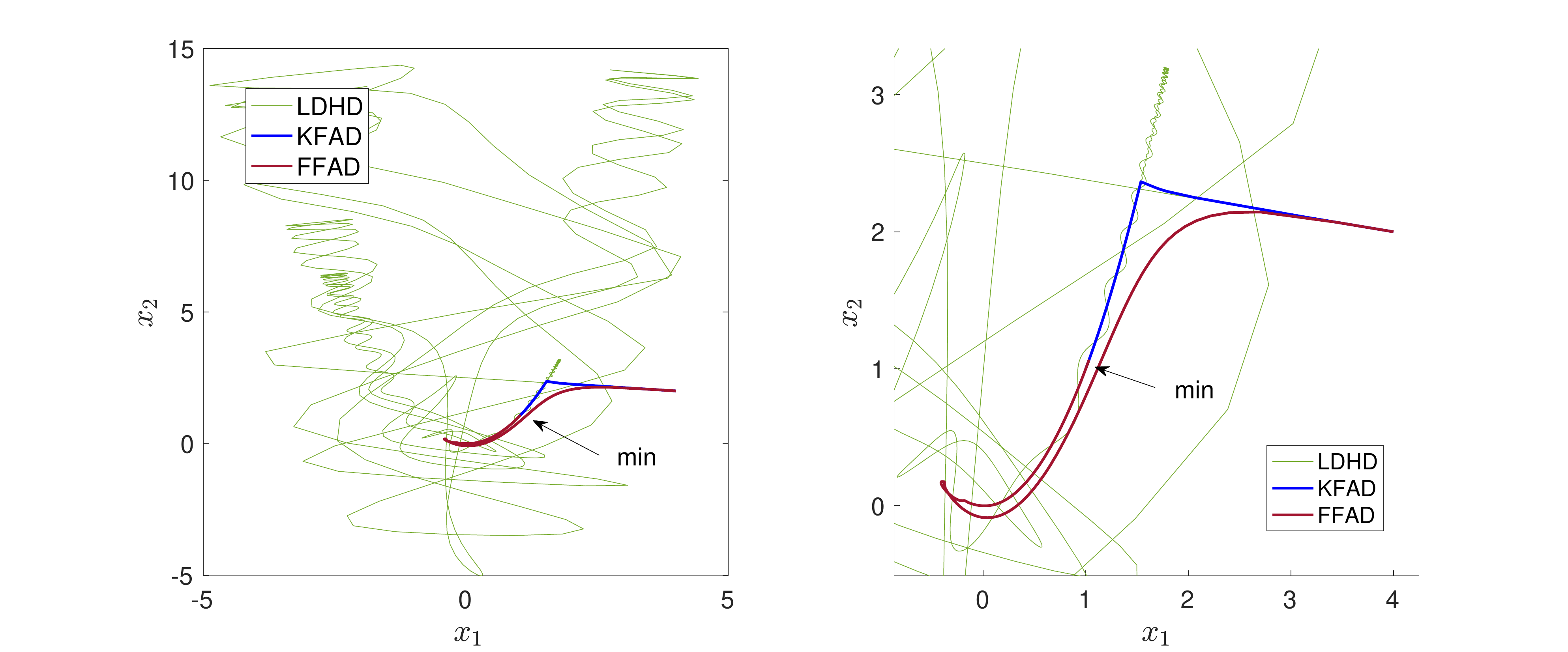}
  \caption{A second illustration of the relative performance of LDHD, KFAD and FFAD on the 2d Rosenbrock model, with initial condition $(4,2)$.  A close-up of the vicinity of the minimum is shown in the right panel. }
  \label{FFADvKFADv2}
  \end{center}
  \end{figure}
In Fig.~\ref{FFADvKFADv2}, the timesteps needed to reach the minimum were as follows: LDHD: 2010, KFAD: 1604, and FFAD: 3658.  
FFAD showed a tendency to  overshoot the channel and then oscillate during the approach, or to slightly miss the channel and then follow it. The performance of KFAD here is remarkable -- it completely removes the oscillatory initial phase. We will provide more detailed numerical studies on the Rosenbrock model in Sec.~\ref{sec:numerical_schemes}, but the results observed here are characteristic of the differences in the methods. KFAD outperforms LDHD for almost all initial conditions.  

\subsection{Projective mixture coupling \label{sec:McFAD}} 

We attribute the somewhat erratic performance of FFAD (in terms of efficiency) to the degenerate nature of the friction coupling.   We also recognize the potential for $\|F\|$ to be very large, which can lead to excessive variation in the friction coupling force.   We therefore considered alternatives based on  
\[
A(x) = \lambda_1 I + \lambda_2 \Pi^F(x),
\]
where $\Pi^F(x) = \|F(x)\|^{-2} F(x) F(x)^T$ is a projector onto the range of $F(x)$.
The use of such a projector can be seen as a zero temperature version of ``projective thermostatting'' \cite{JiaLe2005}, albeit with a novel choice of the projection which has not been previously considered in the thermostat setting.  We explore different choices of the parameters $\lambda_1$, $\lambda_2$ in the numerical experiments.

We illustrate the potential for mixture coupling (McFAD) on the Rosenbrock example, see Fig.~\ref{McFAD}.  The first figure shows comparisons of the minimization curves; the second graphs the distance from the minimum.

\begin{figure}[htbp!]
\begin{center}
  \includegraphics[scale=0.25]{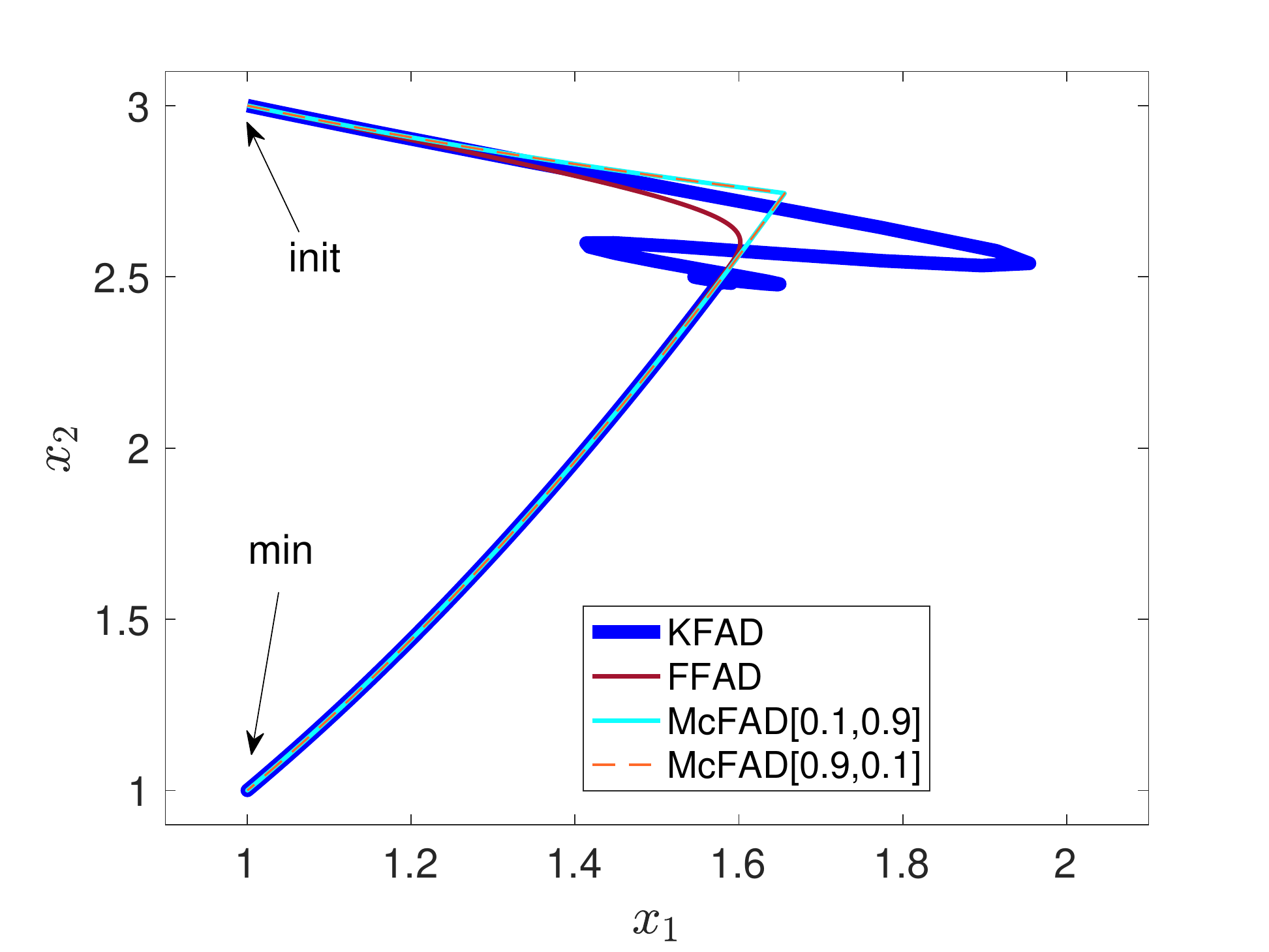} \includegraphics[scale=0.25]{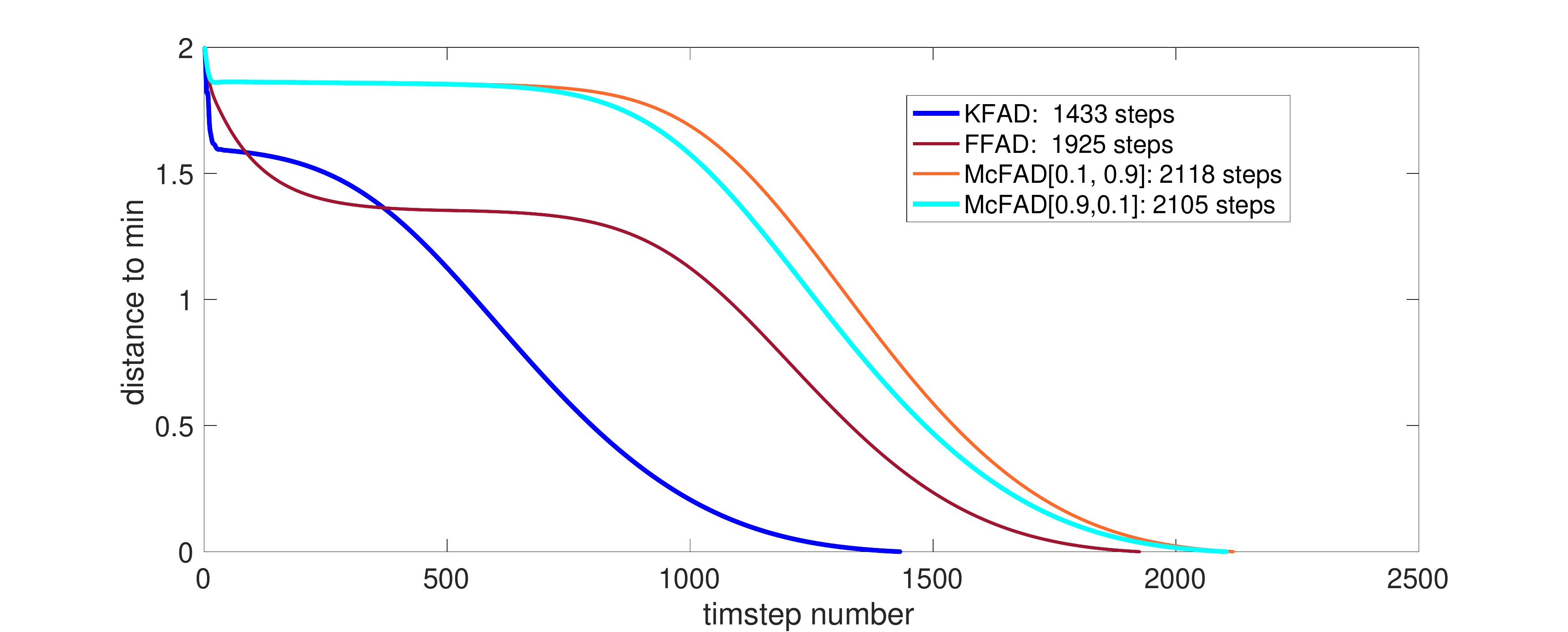}
  \caption{Left: convergence for mixture coupling with parameters~$[\lambda_1,\lambda_2]=[0.1, 0.9]$ and with parameters $[0.9,0.1]$ for the Rosenbrock function.  Comparing to KFAD we see that the convergence is more direct, with essentially no oscillation. Right: we illustrate the convergence in terms of decay of the distance with time.   }
  \label{McFAD} 
  \end{center}
  \end{figure}

  There are wide plateaus in the convergence of FFAD and mixture coupling which slow convergence. KFAD makes more rapid progress to the minimum, despite oscillations in the trajectories themselves.

\section{Dynamical systems analysis of friction-adaptive descent}

In this section, we briefly discuss the properties of the dynamics and introduce a Lyapunov function to show
exponential convergence under certain assumptions on $f$ and $A$, in particular that $f$ is strongly convex and that $A$ is symmetric positive semidefinite.  

Before turning to the convergence theorem, we present some computed ``phase portraits'' of KFAD and FFAD (See Fig.~\ref{phase_portraits}).  By phase portraits we mean projections of the solution curves of each system into some suitable plane (in this case, we consider the $x_1x_2$ projection of FAD in the case of the Rosenbrock function, where we assumed initial $\xi_0=0$).  The parameter $\gamma$ was fixed at 1 and $\alpha$ and $\mu$ were also taken to be 1.  The stepsize for integration was fixed at 0.01.  Points are generated on a grid with unit spacing in $[-4,4] \times [-4,4]$.  In all cases every solution shown converges to $(1,1)$.
\begin{figure}[tbph!]\label{fig:phase_portraits_FFAD_KFAD}
\begin{center}
  \includegraphics[width=2.75in]{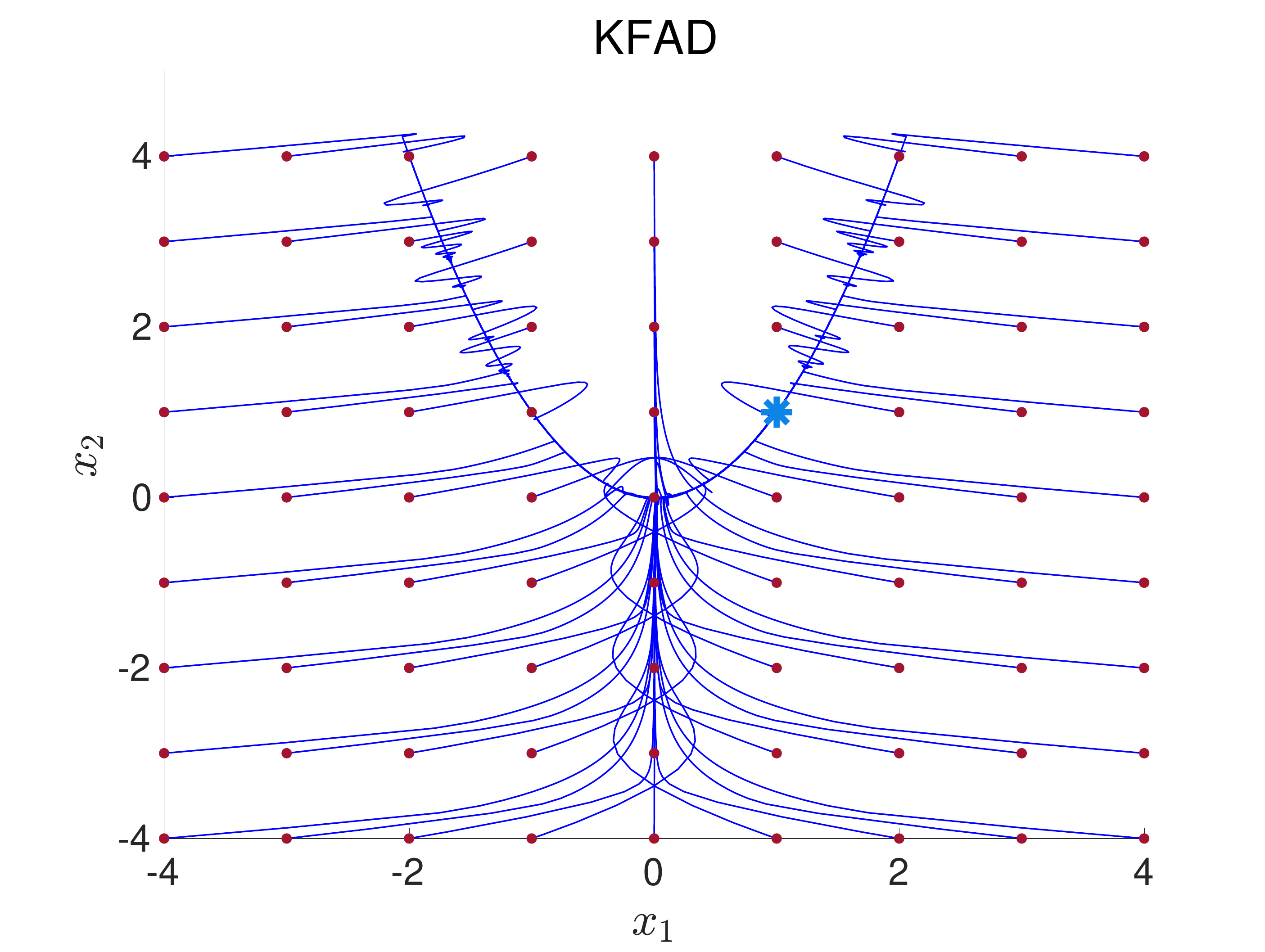}
  \includegraphics[width=2.75in] {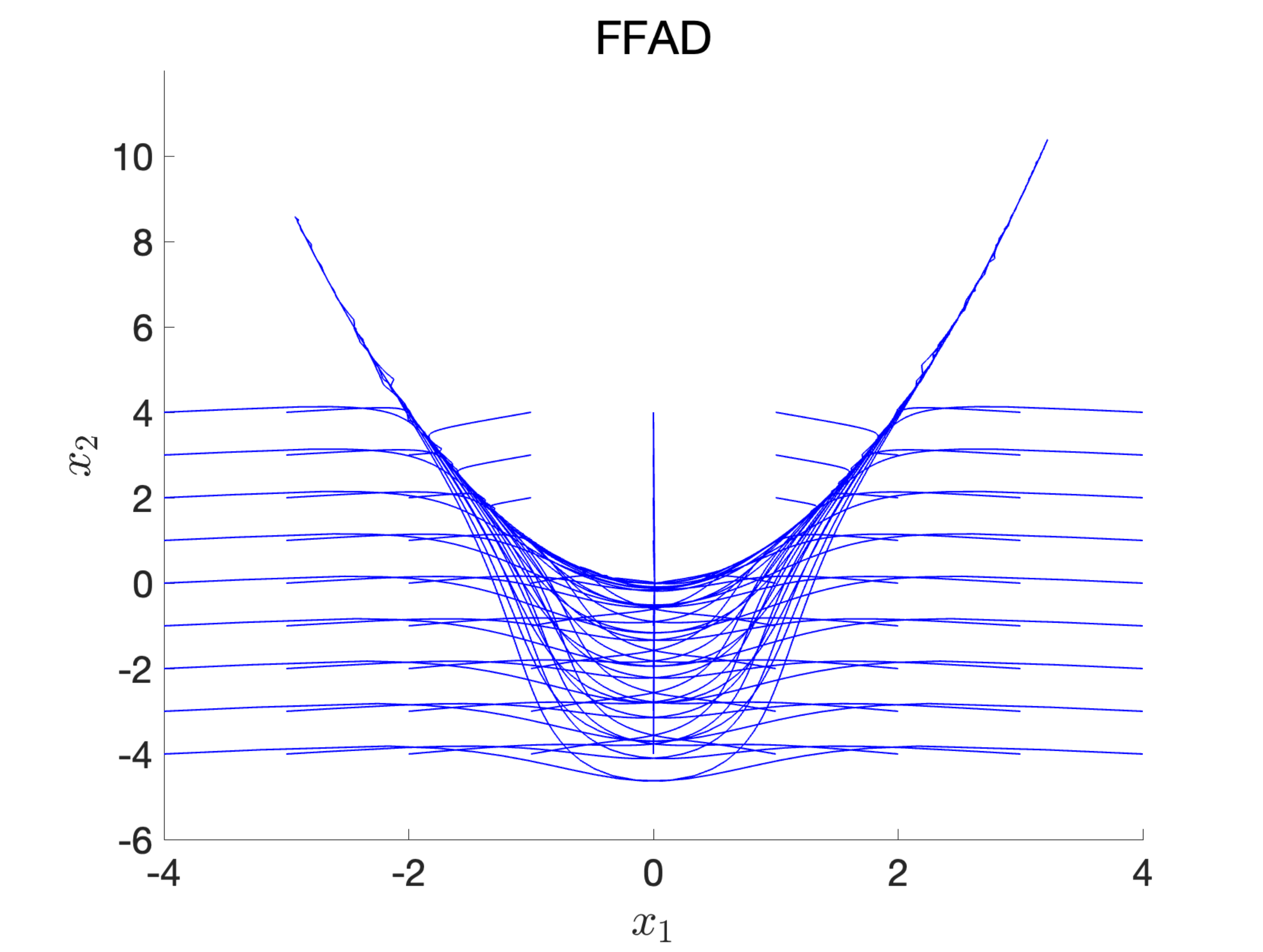}\\
  \caption{Phase portraits, as described in the text, show the convergence of each method from a number of different initial points (red points) in the $x_1x_2$ plane to the minimizer at $(1,1)$, marked as a blue star. Left: KFAD. Right: FFAD
  Note the different scale in the FFAD figure due to the long excursions along the narrow channel.}
  \label{phase_portraits}
  \end{center}
  \end{figure}

Convergence to the minimum is achieved for all the initial conditions, although the way that convergence is achieved depends strongly on the initial point.

\subsection{Elementary properties of friction-adaptive dynamics}

With $\xi$ defined as in (\ref{eq:FADa})-(\ref{eq:FADc}), we have, 
from  $\dot{\xi} = p^T A(x)p/\mu - \alpha \xi$, that
\begin{equation}
\label{eq:xi_integrated}
\xi(t) = \rme^{-\alpha t} \xi(0) + \frac1\mu \int_0^t \rme^{-\alpha (t-s)} p(s)^T A(x(s)) p(s) \, {\rm d} s.
\end{equation}
In particular this shows that, when~$\xi(0)\geq 0$ and $A$ is everywhere symmetric positive semidefinite, then $\xi(t) \geq 0$ for all~$t \geq 0$, thus $\xi$ retains the interpretation of a friction coefficient. 

A first result is that the dynamics is well posed on infinite time horizons under some conditions on~$f$ and~$A$.

\begin{lemma}
\label{lem:bounds}
Assume that~$A(x)$ is symmetric positive semidefinite/ for all~$x \in \R^d$, the function~$f$ is smooth and~$f(x) \to +\infty$ as $\|x\| \to +\infty$.
Then, for any initial condition~$(x_0,p_0,\xi_0) \in \R^d \times \R^d \times \R_+$, the solution of~\eqref{eq:FADa}-\eqref{eq:FADc} is well defined for all times~$t \geq 0$, and there exists~$R \in \R_+$ such that 
\[
\forall t \geq 0, \qquad \|x(t)\| \leq R, \quad \|p(t)\| \leq R, \quad 0 \leq \xi(t) \leq R.
\]
\end{lemma}

\begin{proof}
The Cauchy--Lipschitz theorem ensures the existence and uniqueness of a solution for a positive time. In order to show that the solution is global in time, we introduce the Lyapunov function 
\[
\cG(x,p,\xi) = f(x)-f(x^*) + \frac12 \|p\|^2 + \frac{\mu}{2}\xi^2.
\]
A simple computation shows that the function $\sG(t) = \cG(x(t),p(t),\xi(t))$ satisfies
\[
\begin{aligned}
    \dot{\sG}(t) & = \nabla f(x(t)) \cdot \dot{x}(t) + p(t) \dot{p}(t)  + \mu \xi(t) \dot{\xi}(t) \\ 
    & = -\gamma \|p(t)\|^2 - \alpha \mu \xi(t)^2 \leq 0,
\end{aligned}
\]
where we used that~$\xi(t) \geq 0$. This shows that~$\sG(t) \leq \sG(0)$, from which the result easily follows since~$f-f(x^*)$ is nonnegative. 
\end{proof}

The next result, whose proof is immediate, characterizes equilibria of the dynamics.

\begin{proposition}
\label{prop:equilibria}
The equilibria of the system (\ref{eq:FADa})-(\ref{eq:FADc}) correspond to 
$$(\dot{x}, \dot{p}, \dot{\xi}) 
 = 0 \Rightarrow (p^*, F(x^*) , \xi^*) = 0,$$
i.e. they coincide with the physical equilibria.
 \end{proposition}

The last result relates LDHD and FAD. 
 
 \begin{proposition}
 \label{prop:LDHD_limit}
 Fix a time~$\tau > 0$, and assume that the conditions of Lemma~\ref{lem:bounds}
are satisfied, and that~$A$ is bounded over compact sets. Then, in the limit $\alpha \to \infty$ and $\alpha\mu \to +\infty$, the solution to the system (\ref{eq:FADa})-(\ref{eq:FADc}) converges uniformly over the time interval~$[0,\tau]$ to the solution of  linearly dissipated Hamiltonian dynamics.
 \end{proposition}
 
\begin{proof}
The equality~\eqref{eq:xi_integrated} combined with the bound from Lemma~\ref{lem:bounds} implies that there exists~$K \in \mathbb{R}_+$ (depending on the initial condition) such that
\[
0 \leq |\xi(t)| \leq \rme^{-\alpha t} |\xi(0)| + \frac{K}{\mu} \int_0^t \rme^{-\alpha(t-s)} \, {\rm d} s \leq \rme^{-\alpha t} |\xi(0)| + \frac{K}{\alpha\mu}.
\]
The two terms converge to~0 uniformly on bounded time intervals of the form~$[t_{\rm min},\tau]$ for any~$t_{\rm min} \in (0,\tau)$, respectively as $\alpha\to+\infty$ for the first one, and as~$\alpha\mu \to +\infty$ for the second one. This bound is next combined with the integral formulation of the solution
\[
p(t) = p(0) - \gamma \int_0^t p(s) \, ds - \int_0^t \xi(s) A(x(s)) p(s) \, ds.
\]
The last term converges to~0, which leads to the claimed convergence result since the system (\ref{eq:FADa})-(\ref{eq:FADb}), with~$\xi(t)$ replaced by~0, is precisely linearly dissipated Hamiltonian dynamics.
\end{proof}
The fact that FAD methods eventually inherit the dissipative property of LDHD is reassuring, but in practice this behavior may not be observed
until very late in the descent process.  As we hinted at earlier, the FAD dynamics, even without friction, can be  convergent.  We explore this property in the numerical experiments of Section~\ref{sec:num_exper}, in the context of the Morse cluster.
\subsection{Exponential convergence of the FAD dynamics}
Under additional conditions on~$f$, for instance~$f$ strongly convex, one can obtain the exponential convergence to the equilibria of Proposition~\ref{prop:equilibria}.

\begin{theorem} \label{thm:convergence}
 Assume that~$\gamma > 0$, that $A(x)$ is symmetric positive semidefinite for all~$x \in \R^d$ and bounded over compact sets, and that there exist~$a,b > 0$ such that
 \begin{equation}
\label{eq:condition_f_exp_cv}
a\left[f(x)-f(x^*)\right] + b \|x-x^*\|^2 \leq (x-x^*)\cdot \left( \nabla f(x)-\nabla f(x^*) \right).
 \end{equation}
 Then, for any initial condition~$(x_0,p_0,\xi_0) \in \R^d \times \R^d \times \R_+$, there exist~$\kappa > 0$ and $C \in \R_+$ such that the solution of~\eqref{eq:FADa}-\eqref{eq:FADc} satisfies
 \[
 f(x(t))-f(x^*) + \|p(t)\| + \xi(t) \leq C \rme^{-\kappa t}.
 \]
\end{theorem}

The condition~\eqref{eq:condition_f_exp_cv} is satisfied for~$f \in C^2$ with $0 < m \leq \nabla^2 f(x) \leq M < +\infty$. A similar condition is considered in~\cite{mattingly2002ergodicity}. 

\begin{proof}
Consider~$\varepsilon>0$ and introduce the following Lyapunov function (which is, up to the additional term~$\xi^2$, a  common choice for stochastic Langevin dynamics~\cite{mattingly2002ergodicity}, also considered for dissipated Hamiltonian dynamics~\cite{mo2022}):
\begin{equation}
\label{eq:Lyapunov_cWe}    
\cWe(x,p,\xi) = f(x)-f(x^*) + \frac12 \|p\|^2 + \frac{\mu}{2}\xi^2 + \varepsilon (x-x^*)\cdot p + \varepsilon \|x(t)-x^*\|^2.
\end{equation}
A discrete Cauchy--Schwarz inequality implies, for~$\varepsilon \in [0,1/2]$, the lower bound
\begin{equation}
\label{eq:lower_bound_sWe}
\cWe(x,p,\xi) \geq f(x)-f(x^*) + \frac14 \|p\|^2 + \frac{\mu}{2} \xi^2 + \frac{\varepsilon}{2} \|x(t)-x^*\|^2,
\end{equation}
as well as the upper bound
\begin{equation}
\label{eq:upper_bound_sWe}
\cWe(x,p,\xi) \leq f(x)-f(x^*) + \frac34 \|p\|^2 + \frac{\mu}{2} \xi^2 + \frac{3\varepsilon}{2} \|x(t)-x^*\|^2.
\end{equation}
A simple computation shows that the function $\sWe(t) = \cWe(x(t),p(t),\xi(t))$ satisfies
\[
\begin{aligned}
    \dot{\sWe}(t) & = \nabla f(x(t)) \cdot \dot{x}(t) + p(t) \dot{p}(t)  + \mu \xi(t) \dot{\xi}(t) + \varepsilon p(t)\cdot \dot{x}(t) \\
    & \qquad + \varepsilon \dot{p}(t) \cdot (x(t)-x^*) + 2\varepsilon \dot{x}(t) \cdot (x(t)-x^*)\\ 
    & = -(\gamma-\varepsilon) \|p(t)\|^2 - \alpha \mu \xi(t)^2 \\
    & \qquad - \varepsilon (x(t)-x^*)\cdot \nabla f(x(t)) + \varepsilon (x(t)-x^*) \cdot \left[2-\gamma-\xi(t)A(x(t))\right] p(t) \\ 
    & \leq -(\gamma-\varepsilon) \|p(t)\|^2 - \alpha \mu \xi(t)^2 - \varepsilon (x(t)-x^*)\cdot \nabla f(x(t)) \\
    & \qquad + \varepsilon (2+\gamma+R\|A(x(t))\|) \left( \delta \|x(t)-x^*\|^2 + \frac{1}{4\delta} \|p(t)\|^2\right),
\end{aligned}
\]
where we used a Cauchy--Schwarz inequality and Lemma~\ref{lem:bounds}. Note that~$\|A(x(t))\| \leq \mathcal{A}<+\infty$ since~$A$ is bounded on compact sets and~$x(t)$ is uniformly bounded. We then set~$\delta = b/[2(2+\gamma+R \mathcal{A})]$ 
in order to use~\eqref{eq:condition_f_exp_cv}. In view of the upper bound~\eqref{eq:upper_bound_sWe}, this leads to
\[
\begin{aligned}
\dot{\sWe}(t) & \leq -\left(\gamma-\varepsilon\left[1+\frac{(2+\gamma+R\mathcal{A})^2}{2b}\right]\right) \|p(t)\|^2 - \alpha \mu \xi(t)^2 - \frac{b\varepsilon}{2} \|x(t)-x^*\|^2 \\
& \qquad - a \varepsilon (f(x(t))-f(x^*)) \\
& \leq -\min\left\{2\alpha,a\varepsilon,\frac{b}{3},\frac43\left(\gamma-\varepsilon\left[1+\frac{(2+\gamma+R\mathcal{A})^2}{2b}\right]\right)\right\} \sWe(t),
\end{aligned}
\]
from which one can conclude to the exponential convergence of~$\sWe(t)$ to~0 by a Gronwall inequality, upon choosing~$\varepsilon >0$ small enough (which can be optimized in order to get an explicit lower bound on the convergence rate). This implies the desired inequality.
\end{proof}

Note that the convergence rate a priori depends on the chosen initial condition.  This is due to a technical limitation in the way we prove  exponential convergence, since the dynamics is compared to a family of dissipated Hamiltonian dynamics, of which we consider the worse case scenario. More importantly, our proof currently only holds for~$\gamma > 0$. We believe that the convergence result can be improved, and extended to more general situations. This is however left for future work.
%
%
%
%
\section{Numerical methods}
\label{sec:numerical_schemes}

The dynamical system (\ref{eq:FADa})-(\ref{eq:FADc}) may be unfamiliar unless one has prior experience of Nos\'{e} dynamics.  Since a reliable numerical method is essential in situations where many thousands of steps may be taken to determine a minimum, we discuss the integration method in detail here.  We first mention that in our experiments, we adopted a simple splitting 
integrator to propagate linearly dissipated Hamiltonian dynamics.  Following the standard method for such systems, we break the system into three parts:  (i) $\dot{x} = p$, $\dot{p} = 0$, (ii) $\dot{x} = 0$, $\dot{p} = -\nabla f(x)$, and (iii) $\dot{x} = 0$, $\dot{p} = -\gamma p$.  Each of these systems is easily integrated so we compose the flows to obtain a simple (and reliable) numerical scheme.  The same general procedure can be used to integrate the dynamical equations of friction-adaptive descent.

\subsection{Splitting scheme for Linearly Dissipated Hamiltonian Dynamics (LDHD)}\label{sec:LDHD_splitting}

We discretize the system of equations (\ref{eq:lindiss-1})-(\ref{eq:lindiss-2}) using the splitting scheme presented in (\ref{eq:LDHD-splitting}). Splitting methods are applicable when the dynamics we are trying to solve can be split into a sum of parts each of which is easier to integrate than the original dynamics. For an introduction to splitting methods we refer the reader to \cite{mclachlan2002splitting}. Here, the dynamics is decomposed as
\begin{equation}\label{eq:LDHD-splitting}
\begin{pmatrix}\dot{x} \\ \dot{p}  \end{pmatrix}
 =\underbrace{\begin{pmatrix} p\\0 \end{pmatrix}}_{\text{A}} + 
  \underbrace{\begin{pmatrix}0\\F(x) \end{pmatrix}}_{\text{B}} +
  \underbrace{\begin{pmatrix}0\\- \gamma p \end{pmatrix}}_{\text{D}} .
\end{equation}
We use a symmetric integrator which we denote as ``BADAB''. For each full step of the numerical integrator, we want to solve each part (A,B,D) of the dynamics over a time step~$\delta t$. Since the letter ``B'' appears twice in the integrator, each of the ``B'' steps is integrated with a stepsize of $\delta t/2$. The same holds for the ``A'' step, but not for the ``D'' step, which will be integrated over with a full time step $\delta t$. This is the integrator that we use for all numerical experiments involving LDHD. More details on solving A,B,D are given below. 

\subsection{A splitting scheme for friction-adaptive descent}
\label{sec:first_splitting}

One possible additive decomposition of the FAD dynamics is the following:
\begin{equation}\label{eq:FAD-splitting}
\begin{pmatrix}\dot{x} \\ \dot{p} \\ \dot{\xi} \end{pmatrix}
 =\underbrace{\begin{pmatrix} p\\0\\0 \end{pmatrix}}_{\text{A}} + 
  \underbrace{\begin{pmatrix}0\\F(x)\\0 \end{pmatrix}}_{\text{B}} +
  \underbrace{\begin{pmatrix}0\\-\xi A(x) p\\ \frac{p^T A(x)p}{\mu} - \alpha \xi \end{pmatrix}}_{\text{C}}+
  \underbrace{\begin{pmatrix}0\\- \gamma p\\ 0 \end{pmatrix}}_{\text{D}} .
\end{equation}
To design an algorithm, we combine the steps defined above to create an integration scheme for our method. We discuss another decomposition in Section~\ref{sec:alternative_splitting} below.

In general ABCD represents an integrator where A is solved for time $\delta t$ then B for time $\delta t$, etc.  Since the steps do not commute, we get many different numerical methods on the alphabet A,B,C,D.  Guided by past experience we favor schemes which are symmetric in form (palindromic letter sequences).  

There is an important efficiency constraint on the arrangement of steps.   Since both B and C involve the gradient of~$f$, if we separate these by any step that changes $x$ we will need to perform a second gradient calculation. Since the gradient is very likely the most significant cost in the integrator, it is important to avoid this extra calculation. This suggests to favor schemes that keep B and C together.   It makes sense to combine them at the middle of the step, particularly where gradients are likely to be based on subsampling, thus we suggest to use the integrator  DABCBAD, 
where each evaluation of A,B, and D would be performed for a half timestep. 
With the exception of part C, each of the terms can be integrated exactly. The part C itself can be solved by a further round of splitting.

Let us first consider the A, B and D parts.  The equations of the A part are 
\begin{align*}
\dot{x} & = p,\\
\dot{p} & = 0,\\
\dot{\xi} & = 0.
\end{align*}
The variables~$p$ and $\xi$ are constant during this step and the solution at time~$t$, given initial conditions $(x_0,p_0,\xi_0)$, is just $x(t) = x_0+t p_0$.  In a similar way the positions and the auxiliary variable are constant during the B step, so the force~$F$ is constant, and the update of momenta is therefore $p(t) =p_0+t F(x_0)$.   Because of the simple choice of linear dissipation, the D update boils down to $p(t) = \exp(-t\gamma)p_0$.  This leaves just the C part of the splitting to discuss.

The equations for C are 
\begin{align}
    \dot{p} & = -\xi A(x)p, \label{eq:c1}\\
    \dot{\xi} & = \frac{p^T A(x)p}{\mu} -\alpha \xi. \label{eq:c2}
\end{align}
During this step $\dot{x}=0$ so $x$ is constant and we can write $A(x)\equiv A$ (constant).
We propose to solve this system by one of the following techniques, depending on the method choice.
%
\subsubsection{The ``C'' step for KFAD/FFAD/mixture coupling}
The solution is greatly simplified in certain cases, e.g. $A=I$ (KFAD) or $A=FF^T$ (FFAD). More generally, we can solve it readily for any mixture coupling of the form
\[
A = \lambda_1 I+ \lambda_2 FF^T,
\]
which addresses both of the two special cases as well as more general couplings.

We rely on the following leapfrog scheme:
\begin{align}
p_{n+1/2} & = \exp(-(\delta t/2)\xi_n A) p_n,\\
\xi_{n+1} & = \exp(-\delta t\alpha) \xi_n + \frac{1}{\alpha \mu} (1-\exp(-\delta t\alpha)) p_{n+1/2}^T A p_{n+1/2}, \label{eqn:middle}\\
p_{n+1} & = \exp(-(\delta t/2)\xi_{n+1} A) p_{n+1/2},
\end{align}
To calculate the matrix exponentials we use the fact that $A$ can be written as 
\[
A = \lambda_1 I + \lambda_2 \|F\|^2 \Pi,
\]
where 
\[
\Pi = \|F\|^{-2} F F^T 
\]
is an orthogonal projector. Rodrigues' formula gives, for~$\tau \in \mathbb{R}$,
\[
\exp(\tau \Pi) = I + (\rme^{\tau}-1) \Pi.
\]
Thus, for any $p \in \mathbb{R}^d$,
\[
\exp(\tau \Pi) p = p + (\rme^{\tau}-1) \Pi p = p+ (\rme^{\tau}-1) \|F\|^{-2} (p\cdot F) F.
\]
This gives, for $\tau_n = -\delta t\lambda_2 \|F\|^2 \xi_n/2$,
\[
p_{n+1/2}  = \exp(-(\delta t/2) \xi_n \lambda_1) \left ( p_n+ (\rme^{\tau_n}-1)\|F\|^{-2} (p_n \cdot F) F \right ).
\]
After computing $\xi_{n+1}$ using (\ref{eqn:middle}), we then follow by computing $\tau_{n+1}= -\delta t\lambda_2 \|F\|^2 \xi_{n+1}/2$ and setting
\[
p_{n+1}  = \exp(-(\delta t/2) \xi_{n+1} \lambda_1) \left ( p_{n+1/2} + (\rme^{\tau_{n+1}}-1)\|F\|^{-2} (p_{n+1/2} \cdot F) F \right ).
\]
%
\subsubsection{The ``C'' step for general matrices $A(x)$}
\label{sec:C_step_general_A}
In the general case we cannot simplify the matrix exponentials as above. Instead we suggest to use the following linearly implicit scheme (here $A\equiv A(x)$ is the matrix fixed at the start of the C step):
\[
\begin{aligned}
p_{n+1/2} & = p_n - (\delta t/2) \xi_n A p_{n+1/2},\\
\xi_{n+1} & = \rme^{-\alpha \delta t/2} \xi_n + (\mu\alpha)^{-1} (1- \exp(-\alpha \delta t) ) p_{n+1/2}^TAp_{n+1/2} ,\\
p_{n+1} & = p_{n+1/2} - (\delta t/2) \xi_{n+1} A p_{n+1/2}.
\end{aligned}
\]
%
\subsection{Alternative solution for higher accuracy}
\label{sec:alternative_splitting}
Another possible splitting scheme for the FAD dynamics is based on the following decomposition of the dynamics, which differs from the one considered in Section~\ref{sec:first_splitting} in the way the C,D parts are decomposed:
\begin{equation}\label{eq:FAD-splitting2}
\begin{pmatrix}\dot{x} \\ \dot{p} \\ \dot{\xi} \end{pmatrix}
 =\underbrace{\begin{pmatrix} p\\0\\0 \end{pmatrix}}_{\text{A}} + 
  \underbrace{\begin{pmatrix}0\\F(x)\\0 \end{pmatrix}}_{\text{B}} +
  \underbrace{\begin{pmatrix}0\\-\xi A(x) p\\ \frac{p^T A(x)p}{\mu} \end{pmatrix}}_{\text{C'}}+
  \underbrace{\begin{pmatrix}0\\- \gamma p\\ -\alpha \xi \end{pmatrix}}_{\text{D'}} .
\end{equation}
We derive the integrator for C' here as the other parts of the splitting scheme can be integrated analytically, as discussed in Section~\ref{sec:first_splitting}.  As for the previous splitting, during the evolution of the C' system, the matrix $A(x)\equiv A$ (i.e., is fixed).  There are several cases to consider.

\par\noindent {\bf Case I.}  If $A$ is a projector, $A^2=A$,  we proceed as follows.
Taking the dot product of~\eqref{eq:c1} by~$p$, we find equations for $\xi$ and $\omega = p^T Ap$ of the form
\begin{align}
\dot{\omega } & = -2\omega \xi, \label{eq:xiomegaeq1}\\ 
\dot{\xi}  & =  \omega/\mu . \label{eq:xiomegaeq2}
\end{align}
Once $\xi$ has been found, the solution of (\ref{eq:c1}) requires solving 
\[
\dot{p}(t) =  -\xi(t) A p. 
\]
In the case that $A$ is the identity (KFAD) the calculation is trivial.  It also easily solved in other settings since the time-dependency in the right hand side is scalar, in other words~$\xi(t) A$ commutes with $\xi(t') A$ for any $t,t' \geq 0$.  In this case the solution is just
\begin{equation}
p(t) = \exp\left ( \eta(t) A \right ) p(0),\label{eq:peqalt2}
\end{equation}
where $\eta(t) = -\int_0^t \xi(s) {\rm d} s$.
We can compute the exponential in (\ref{eq:peqalt2}) using Rodrigues' formula.

The equations (\ref{eq:xiomegaeq1})-(\ref{eq:xiomegaeq2}) are completely integrable, but their exact solution requires the integration of transcendental functions and their subsequent inversion. 
Instead, we solved the two-dimensional system by further splitting and integration with smaller
timesteps, in a multiple time-stepping approach~\cite{Tu1992}.

This scheme can be implemented with any desired accuracy.  We used this integrator with good results in our tests of atomic clusters in Section 6.

\par\noindent {\bf Case II.} We next consider $A = \lambda_1 I + \lambda_2 \Pi$, where $\Pi$ is a projector.  Then we can define $\theta = \|p\|^2$ and $\omega = p^T \Pi p$.  After some calculation, this results in a closed system of three equations in the three variables $\theta$, $\omega$, and $\xi$.  These can be solved to high accuracy using splitting.

Another special case arises if $A$ is either diagonal or easily diagonalized by some orthogonal matrix, that is $A=XDX^T$ where $D$ is diagonal. Then  splitting and multiple timestepping gives a high accuracy solution for little additional work.
%
%
%
\section{Convergence analysis}
%
To discuss convergence to the minimizer in the case of the discretization scheme, we first observe that FAD methods can be viewed as dissipated versions of Hamiltonian dynamics.  In the simplest case we consider the system as being evolved under a Verlet discretization in combination with steps that introduce linear damping and nonlinear controlled damping (through auxiliary equations). 
The analysis we provide in this section proceeds in two steps: we first check whether stationary configurations of discretized dynamics indeed correspond to critical points of the continuous dynamics; and then proceed to prove longtime convergence of the discrete dynamics to the equilibria specified in Proposition~\ref{prop:equilibria}. 

The analysis is conducted for a simple first order splitting scheme. A similar analysis can be conducted for other schemes, such as the ones derived in Section~\ref{sec:numerical_schemes}. We however consider here yet another scheme, simpler in its final form than the schemes considered in Section~\ref{sec:numerical_schemes}, to simplify the presentation. 

More precisely, we consider the C'D'BA splitting, where the C' part is integrated by analytically updating the $p$ variable with $\xi A(x)$ fixed, and then adjusting the $\xi$ variable to preserve the quadratic invariant $\|p\|^2+\mu \xi^2$ (associated to the subdynamics~C'). This scheme is slightly simpler to analyze than the schemes of the previous section. It corresponds to the following sequence of updates:
\begin{align}
   \wp_{n+1/2} & = \rme^{-A(x_n)\xi_n\Dt} p_n, \label{eq:cdba_1}\\
   \wxi_{n+1} & = \xi_n \sqrt{1 + p_n^T \frac{I-\rme^{-2A(x_n)\xi_n\Dt}}{\mu \xi_n^2} p_n}, \label{eq:cdba_2}\\
   \wp_{n+1} & = \rme^{-\gamma \Dt}\wp_{n+1/2}, \label{eq:cdba_3}\\
   \xi_{n+1} & = \rme^{-\alpha \Dt} \wxi_{n+1}, \label{eq:cdba_4}\\
   p_{n+1} & = \wp_{n+1} - \Dt \nabla f(x_n), \label{eq:cdba_5}\\
   x_{n+1} & = x_n + \Dt p_{n+1}. \label{eq:cdba_6}    
\end{align}
Upon introducing the matrix $\beta_{n,\Dt} = \rme^{-(\gamma I+\xi_n A(x_n))\Dt} \in \R^{d \times d}$, this scheme can be rewritten as
\begin{align}
   p_{n+1} & = \beta_{n,\Dt} p_n - \Dt \nabla f(x_n), \label{eq:cdba_s1}\\
   x_{n+1} & = x_n + \Dt \beta_{n,\Dt} p_n - \Dt^2 \nabla f(x_n), \label{eq:cdba_s2}\\
   \xi_{n+1} & = \rme^{-\alpha \Dt} \xi_n \sqrt{1 + p_n^T \frac{I-\rme^{-2A(x_n)\xi_n\Dt}}{\mu \xi_n^2} p_n}. \label{eq:cdba_s3}
\end{align}

Before proceeding to a convergence analysis of the discrete dynamics, we examine the {\em regularity} of the numerical method, that is, whether the equilibria of the discretized dynamics coincide with those of the continuous dynamics. We wish to show that the equilibrium for the continuous dynamics $(x^*, p^*, \xi^* )$ derived in Proposition~\ref{prop:equilibria} is also an equilibrium for the discretised system, and further to ensure that the proposed discrete dynamics does not introduce ``artificial" stationary points which do not correspond to equilibria of the continuous system.
\begin{proposition}
\label{prop:regularity_cv_discretized}
A state $z = (x, p, \xi)$ is an equilibrium point of the discretized dynamics C'D'BA as presented in equations (\ref{eq:cdba_1})-(\ref{eq:cdba_6}) if and only if $z$ is an equilibrium of the continuous dynamics given by equations (\ref{eq:FADa})-(\ref{eq:FADc}).
\end{proposition}
\begin{proof}
    Let $z^* = (x^*, p^*, \xi^*)$ be an equilibrium point of the continuous dynamics (\ref{eq:FADa})-(\ref{eq:FADc}). We have shown in Proposition~\ref{prop:equilibria} that such an equilibrium is of the form  $z^* = (x^*,0,0 )$, where $\nabla f(x^*) = 0$.
    
    It is trivial to see that if we start with $z_n = (x^*,0,0)$ and take one step of the discretized dynamics, i.e. apply the scheme (\ref{eq:cdba_1})-(\ref{eq:cdba_6}) , we obtain $z_{n+1} = (x^*,0,0)$. 

    We next show that any equilibrium point of the discrete dynamics is an equilibrium of the continuous dynamics.
    Consider an equilibrium point $z_n$  of the system of equations (\ref{eq:cdba_1})-(\ref{eq:cdba_6}). Then, starting from $z_n$ and taking a step of the discretised dynamics, one remains at the equilibrium, i.e. $z_{n+1} = z_n$.   Setting $x_{n+1} = x_n$ in~\eqref{eq:cdba_6} gives us $p_{n+1}=0$. This implies $p_n=p_{n+1}=0$ as~$z_n$ is an equilibrium point. Using~$p_n = 0$, \eqref{eq:cdba_s1} yields $\nabla f(x_n) = 0$.
    Finally, we reformulate~\eqref{eq:cdba_s3} as 
    \[
    \xi_{n+1} = \rme^{-\alpha \Dt}\sqrt{\xi_n^2 + p_n^T \frac{I-\rme^{-2A(x_n)\xi_n\Dt}}{\mu } p_n}.
    \]
    Using that $\xi_{n+1} = \xi_n$ and $p_n = 0$, we obtain $\xi_{n}^2 = e^{-2a \delta t} \xi_n^2 \Rightarrow \xi_n = 0$ since $\alpha, \delta t \neq 0$. 
    Therefore, we have that $z_n = z^*$.
\end{proof}
\begin{proposition}
\label{prop:cv_discretized}
Assume that $f \in C^2$ with $0 < m \leq \nabla^2 f(x) \leq M < +\infty$, that~$A$ is a function with values in the space of symmetric positive semidefinite $d \times d$~matrices, bounded on compact sets, and that~$\gamma > 0$. Fix~$L \in \mathbb{R}_+$. Then, for any initial condition~$(x_0,p_0,\xi_0) \in \R^d \times \R^d \times \R_+$ such that
\[
\|x_0 \|+\|p_0\|+\xi_0 \leq L,
\]
there exist~$\Dt^\star>0$, $\kappa>0$ and~$C \in \R_+$ (depending on~$L$) such that
\[
\forall \Dt \in (0,\Dt^\star), \quad \forall n\geq 0, \qquad f(x_n) -f(x^*) + \|p_n\|^2 + \xi_n^2 \leq C \rme^{-\kappa n\Dt}.
\]
\end{proposition}

The upper bound on the Hessian could be replaced more generally by $M$-smoothness. The lower bound is equivalent to $m$-strong convexity. 

\begin{proof}
  Without loss of generality, and in order to lighten the notation, we assume that~$x^* = 0$ and~$f(x^*) = 0$. We compute the variations of the various terms appearing in the Lyapunov function~\eqref{eq:Lyapunov_cWe}, singling out terms of order~$\Dt$ and gathering terms of order~$\Dt^2$ and higher. The constant~$C$ which allows to bound the remainder term may change from line to line, but is independent of~$\Dt$ and~$\varepsilon$. First, using the upper bound on the Hessian,
\[
\begin{aligned}
f(x_{n+1}) & \leq f(x_n) + \Dt \nabla f(x_n) \cdot p_{n+1} + \frac{ \Dt^2 M}{2} \|p_{n+1}\|^2 \\
& = f(x_n) + \Dt \beta_{n,\Dt} p_n \cdot \nabla f(x_n) + \Dt^2 \left[\frac{M}{2} \|\beta_{n,\Dt}p_n - \Dt \nabla f(x_n)\|^2 - \|\nabla f(x_n)\|^2 \right] \\
& \leq f(x_n) + \Dt \beta_{n,\Dt} p_n \cdot \nabla f(x_n) + C\Dt^2 \left(\|p_n\|^2 + \|\nabla f(x_n)\|^2 \right).
\end{aligned}
\]
Moreover, 
\[
\begin{aligned}
\|x_{n+1}\|^2 & = \|x_n\|^2 + 2 \Dt \beta_{n,\Dt} p_n \cdot x_n + \Dt^2\left[-2x_n \cdot \nabla f(x_n) + p_n^T \beta_{n,2\Dt} p_n \right] \\
& \qquad -2\Dt^3 \beta_{n,\Dt} p_n \cdot \nabla f(x_n) + \Dt^4 \| \nabla f(x_n) \|^2 \\
& \leq \|x_n\|^2 + 2 \Dt \beta_{n,\Dt} p_n \cdot x_n -2\Dt^2 x_n \cdot \nabla f(x_n) + C\Dt^2 \left(\|p_n\|^2 + \|\nabla f(x_n)\|^2 \right).
\end{aligned}
\]
Next,
\[
\begin{aligned}
\|p_{n+1}\|^2 & = p_n^T \beta_{n,2\Dt} p_n - 2 \Dt \beta_{n,\Dt} p_n \cdot \nabla f(x_n) + \Dt^2 \|\nabla f(x_n)\|^2.
\end{aligned}
\]
Additionally,
\begin{align*}
  x_{n+1}\cdot p_{n+1} & = p_n \cdot x_n -(\Id-\beta_{n,\Dt}) p_n \cdot x_n - \Dt x_n \cdot \nabla f(x_n) + \Dt p_n ^T \beta_{n,2\Dt} p_n \\
  &  \qquad -2\Dt^2 \beta_{n,\Dt} p_n \cdot \nabla f(x_n) + \Dt^3 \|\nabla f(x_n)\|^2 \\
  & \leq p_n \cdot x_n - \Dt x_n \cdot \nabla f(x_n) -(\Id-\beta_{n,\Dt}) p_n \cdot x_n + \Dt p_n^T \beta_{n,2\Dt} p_n \\
  & \qquad + C\Dt^2 \left(\|p_n\|^2 + \|\nabla f(x_n)\|^2 \right) \\
  & \leq p_n \cdot x_n - \Dt x_n \cdot \nabla f(x_n) -(\Id-\beta_{n,\Dt}) p_n \cdot x_n + \Dt \rme^{-2\gamma\Dt}\|p_n\|^2 \\
  & \qquad + C\Dt^2 \left(\|p_n\|^2 + \|\nabla f(x_n)\|^2 \right).
\end{align*}
Finally, 
\[
\mu\xi_{n+1}^2 = \rme^{-2\alpha\Dt}\left[\mu\xi_n^2 + p_n^T \left(\Id-\rme^{-2\xi_n A(x_n)\Dt}\right)p_n\right].
\]

We next consider the discrete Lyapunov function
\begin{equation}
\label{eq:Lyapunov_discr}
\cW(x,p,\xi) = f(x) + \frac{\|p\|^2}{2} + a \mu \xi^2 + b \|x\|^2 + c x \cdot p,
\end{equation}
with parameters~$a,b,c$ to be determined. We choose the factor~$1/2$ for~$\|p\|^2$ so that the terms~$\beta_{n,\Dt} p_n \cdot \nabla f(x_n)$ coming from the bounds on~$f(x_{n+1})$ and~$\|p_{n+1}\|^2$ cancel. Then, for~$c$ small enough,
\[
\begin{aligned}
  & \cW(x_{n+1},p_{n+1},\xi_{n+1}) \\
  & \qquad \leq \cW(x_n,p_n,\xi_n) + 2b \Dt \beta_{n,\Dt} p_n \cdot x_n - \frac12 p_n^T (\Id-\beta_{n,2\Dt}) p_n \\
  & \qquad\qquad + c\left[ - \Dt x_n \cdot \nabla f(x_n) -(\Id-\beta_{n,\Dt}) p_n \cdot x_n + \Dt \rme^{-2\gamma\Dt}\|p_n\|^2  \right] \\
  & \qquad\qquad - a \left(1-\rme^{-2\alpha\Dt}\right) \mu\xi_n^2 + a \rme^{-2\alpha\Dt} p_n^T \left(\Id-\rme^{-2\xi_n A(x_n)\Dt}\right)p_n \\
  & \qquad\qquad -2b \Dt^2 x_n \cdot \nabla f(x_n) + C\Dt^2 \left(\|p_n\|^2 + \|\nabla f(x_n)\|^2 \right) \\
  & \qquad = \cW(x_n,p_n,\xi_n) - a \left(1-\rme^{-2\alpha\Dt}\right) \mu\xi_n^2 - (c+2b\Dt)\Dt\, x_n \cdot \nabla f(x_n) \\
  & \qquad\qquad - \left[\frac12 - a \rme^{-2\alpha\Dt} - c \Dt \rme^{-2\gamma\Dt}\right]\|p_n\|^2 + \left[\frac{\rme^{-2\gamma\Dt}}{2}-a\rme^{-2\alpha\Dt}\right] p_n^T \rme^{-2\xi_n A(x_n)\Dt}p_n \\
  & \qquad\qquad + \left(\left[2b \Dt\beta_{n,\Dt} -c(\Id-\beta_{n,\Dt})\right] p_n\right) \cdot x_n + C\Dt^2 \left(\|p_n\|^2 + \|\nabla f(x_n)\|^2 \right).
\end{aligned}
\]
The last equality suggests to choose $a = \rme^{2(\alpha-\gamma)\Dt}/2$.  This allows to remove the factor
$p_n^T \rme^{-2\xi_n A(x_n)\Dt}p_n$. The component proportional to~$c$ in the cross term~$x_n \cdot p_n$ is however not so easy to treat since the crude bound~$\|\Id-\beta_{n,\Dt}\|\|x_n\|\|p_n\|$ would potentially lead to a factor~$\|x_n\|\|p_n\|$ if~$\beta_{n,\Dt}$ is very small, i.e. if~$\xi_n A(x_n)$ is large. We therefore need bounds on~$x_n,\xi_n$ over the iterations in order to make sure that~$\xi_n A(x_n)$ remains uniformly bounded. Such bounds are, in turn, obtained from the estimates on the Lyapunov function.

More precisely, we set $a = \rme^{2(\alpha-\gamma)\Dt}/2$ and $c=b=\varepsilon>0$, so that, using $x \cdot \nabla f(x) \geq 0$, and upon choosing~$\Dt$ and~$\varepsilon$ small enough,
\begin{align} \label{eq:lyapunov_almost_final_inequal}
& \cW(x_{n+1},p_{n+1},\xi_{n+1}) \leq \cW(x_n,p_n,\xi_n) -\varepsilon \Dt \, x_n \cdot \nabla f(x_n) - \frac{\alpha\mu \Dt}{2}\xi_n^2 - \frac{\gamma \Dt}{2} \|p_n\|^2 \nonumber\\
& \qquad \qquad\qquad +\varepsilon\left[2\Dt\beta_{n,\Dt} - (\Id-\beta_{n,\Dt})\right]p_n \cdot x_n + C\Dt^2 \left(\|p_n\|^2 + \|\nabla f(x_n)\|^2 \right),
\end{align}

The assumption $0 < m < \nabla^2 f(x) \leq M$ together with the equalities $f(0)=0$ and $\nabla f(0)=0$ implies through Taylor expansions with integral remainder that~$f(x) \geq m \|x\|^2 / 2$ and~$\|\nabla f(x)\| \leq M \|x\|$. In view of these two inequalities, there exists~$K \in \R_+$ such that, uniformly in~$\varepsilon \in (0,1/2]$,
\[
\|p\|^2 + \|\nabla f(x)\|^2 \leq K \cW(x,p,\xi).
\]

Let us fix~$R>0$. Since $A$ is bounded on compact sets, there exists~$\mathscr{W}_R \in \R_+$ such that, for any~$(x,\xi)$ with~$f(x) + \mu \xi^2 \leq \mathscr{W}_R$, 
\[
\xi A(x) \leq R,
\]
so that 
\begin{equation}
    \label{eq:bound_exp_num_cv}
    \rme^{-\xi A(x)\Dt} \geq \rme^{-R \Dt}. 
\end{equation}
Therefore, if $x$ and $\xi$ are bounded through the inequality~$f(x) + \mu \xi^2 \leq \mathscr{W}_R$, then $\beta_{n,\delta t}$ does not become too small. Indeed, whatever the choice of~$\varepsilon \in (0,1/2]$, and upon choosing~$\Dt$ small enough so that~$a \leq 1$, one obtains from (\ref{eq:Lyapunov_discr}) that $f(x) + \mu \xi^2 \leq \cW(x,p,\xi)$. The bound~$f(x) +\mu \xi^2 \leq \mathscr{W}_R$ therefore holds as long as~$\cW(x,p,\xi) \leq \mathscr{W}_R$, irrespectively of the value of~$\varepsilon \in (0,1/2]$.
In particular, any~$(x,p,\xi) \in \R^d \times \R^d \times \R_+$ with~$\cW(x,p,\xi) \leq \mathscr{W}_R$ satisfies
\begin{equation} \label{eq:beta_bound}
    0 \leq \Id-\rme^{-(\gamma+\xi A(x))\Dt} \leq 1-\rme^{-(R+\gamma)\Dt}.
\end{equation}

We now conclude the proof by induction. We assume that~$\cW(x_n,p_n,\xi_n) \leq \mathscr{W}_R$, with~$R>0$ chosen such that the inequality holds for~$n=0$ as well. We want to prove that a similar inequality holds for~$n+1$. Note that the induction hypothesis implies that the condition~\eqref{eq:beta_bound} holds. Using \eqref{eq:lyapunov_almost_final_inequal} and applying \eqref{eq:beta_bound}, we have
\begin{align*}
\cW(x_{n+1},p_{n+1},\xi_{n+1}) & \leq (1+C K \Dt^2) \cW(x_n,p_n,\xi_n) \\
& \qquad -\varepsilon \Dt \, x_n \cdot \nabla f(x_n) - \frac{\alpha\mu \Dt}{2}\xi_n^2 - \frac{\gamma \Dt}{2} \|p_n\|^2 \\
& \qquad + \varepsilon \left(2\Dt + \left\|\Id-\beta_{n,\Dt}\right\|\right)\|p_n\|\|x_n\| \\
& \leq (1+C K \Dt^2) \cW(x_n,p_n,\xi_n) + \varepsilon \left(2\Dt + 1-\rme^{-(R+\gamma)\Dt}\right) \|p_n\|\|x_n\| \\
& \qquad -\varepsilon \Dt \, x_n \cdot \nabla f(x_n) - \frac{\alpha\mu \Dt}{2}\xi_n^2 - \frac{\gamma \Dt}{2} \|p_n\|^2 \\
& \leq (1+C K \Dt^2) \cW(x_n,p_n,\xi_n) \\
& \qquad - \varepsilon\Dt \left( x_n \cdot \nabla f(x_n) -\eta\frac{\left|2\Dt + 1-\rme^{-(R+\gamma)\Dt}\right|}{\Dt} \|x_n\|^2 \right) \\
& \qquad - \frac{\alpha\mu \Dt}{2}\xi_n^2 - \frac{\Dt}{2}\left[\gamma - \frac{\varepsilon}{2\eta} \frac{\left|2\Dt - 1+\rme^{-(R+\gamma)\Dt}\right|}{\Dt}\right] \|p_n\|^2 \\
& \leq (1-\rho \Dt + CK\Dt^2) \cW(x_n,p_n,\xi_n), 
\end{align*}
for some~$\rho>0$, upon choosing~$\eta,\varepsilon > 0$ small enough (which is possible in view of the discussion following~\eqref{eq:bound_exp_num_cv}). The latter right hand side remains smaller than~$\mathscr{W}_R$ for~$\Dt$ small enough. This shows that the induction hypothesis holds for~$n+1$ as well, and allows to conclude that it holds for all~$n \geq 0$.

In fact, the very same calculations also give the claimed exponential convergence.
\end{proof}
%

%
\section{Numerical experiments}
\label{sec:num_exper}
We concentrate in Section~\ref{sec:num_Rosenbrock} on exploring the roles of the parameters and their ranges in the case of the Rosenbrock problem, and on comparing the performance of FAD methods to linearly dissipated Hamiltonian descent. After this we turn our attention in Section~\ref{sec:num_clusters} to the application of FAD methods to the minimization of atomic clusters.

\subsection{Parameter selection in the Rosenbrock problem}
\label{sec:num_Rosenbrock}

In this section, we examine how different combinations of $\gamma$ and $\Dt$ affect convergence of LDHD as given by (\ref{eq:lindiss-1})-(\ref{eq:lindiss-2}) and mixture couplings according to the dynamics (\ref{eq:FADa})-(\ref{eq:FADb})-(\ref{eq:FADc}) and also examine how the parameters $\mu$ and $\alpha$ affect convergence. All comparisons in this first subsection are performed on the Rosenbrock function. 

When comparing LDHD to FAD schemes, we used, as mentioned, the symmetric integrator for LDHD denoted as BADAB presented in Sec.~\ref{sec:LDHD_splitting}, where D refers to the linear damping.  Contraction rates for Langevin dynamics under synchronous coupling \cite{leimkuhler2023contraction} also establish exponential convergence for BADAB (as many other discretizations) of LDHD.

We start by examining three different specific initializations and comparing the $(\gamma, \Dt)$ plots  for LDHD and four variants of mixture coupling, namely the projective mixture-coupling method of Sec.~\ref{sec:McFAD} with $[\lambda_1, \lambda_2] = [0.1,0.9], [0.5,0.5], [0.9,0.1]$ and $[1,0]$. The three initializations we examine correspond to $(x_0, y_0) = (1.5, -0.5)$, $(x_0, y_0) = (0, -2)$ and $(x_0, y_0) = (4, 4)$.
We recall that the minimum of the two-dimensional Rosenbrock is at $(x^*,y^*) = (1,1)$. The momenta, as well as the auxiliary variable $\xi$ are initialized to zero.  The graphs are colored according to the number of iterations it takes for each of the methods to reach the vicinity of the minimum (within a tolerance of~$10^{-4}$).
Each row of Fig.~\ref{fig:Rosen_gamma_h_spec_inits} corresponds to a specific initialization and each column to one of the five optimizers (LDHD and the four variants of projective mixture coupling). The maximum number of iterations has been set to 3000. The parameters $\mu$ and $\alpha$ were set to $\mu=1$ and $\alpha=1$ for the runs in Fig.~\ref{fig:Rosen_gamma_h_spec_inits}.

\begin{figure}[htbp!]
  \centering
  \includegraphics[width=5.9in]{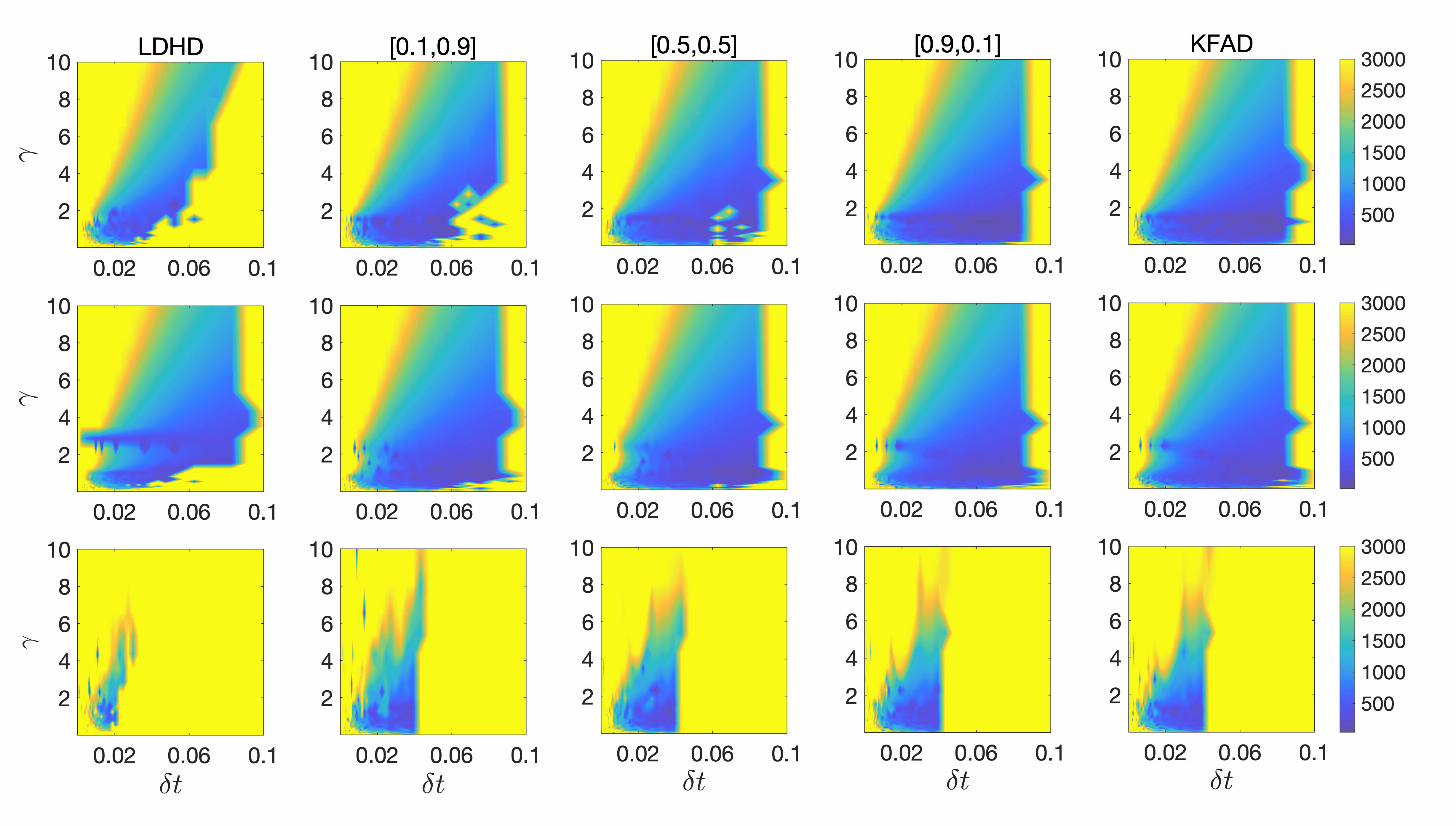}
  \caption{Comparison of number of iterations until convergence for LDHD versus projective mixture coupling. We used three different initial conditions: row 1, IC1: $(x_0,y_0)=(1.5,-0.5)$, row 2, IC2: $(0,-2)$, row 3, IC3: $(4,4)$.  The pair at the top of a column corresponds to $[\lambda_1,\lambda_2]$ in projective mixture coupling.  Each column corresponds to a different method. In all FAD variants, $\mu=\alpha=1$. The maximum number of iterations is set to 3000. The same axes and colormap range have been used for all graphs.}
  \label{fig:Rosen_gamma_h_spec_inits}
\end{figure}

As we can see in Fig.~\ref{fig:Rosen_gamma_h_spec_inits}, for the three specific initializations we examined, KFAD  is much more robust than LDHD in the sense that it allows fast convergence for a much wider range of step sizes and friction parameters. This is especially noticeable in the low friction
regime, where LDHD converges rapidly only for small stepsizes. The results also suggest that projective mixture coupling variants that are close to KFAD are faster and more robust than variants closer to FFAD. It should be noted that performance of all methods for the last initialization $(x_0, y_0) = (4,4)$ is worse, as that initialization is further away from the minimum and the high gradients and steep potential make convergence more challenging.  

Having looked at some specific initializations to get an initial comparison between LDHD and projective mixture coupling variants, we  explore $(\gamma,\Dt)$ plots, but this time averaged over the choice of initial point in the $(x,y)$ domain. In this way, we remove the effect that specific initializations may have on efficiency and can see which method performs better on average. We again set $\mu=1$, $\alpha=1$ for the runs in Fig.~\ref{fig:Rosen_gamma_h_avrg_inits_mu1_a1}. We generate a grid of $(x,y)$ pairs, where $x$ ranges between $(-2,2)$ and $y$ between $(-1,3)$. We create a $(40 \times 40)$ grid, and therefore average over $1600$ different initializations. The plots are again coloured according to the number of iterations  it takes the optimizer to reach the vicinity of the minimum. The maximum number of iterations is set to 1500 for better visualization purposes.

\begin{figure}[htbp!]
  \centering
  \includegraphics[width=5.5in]{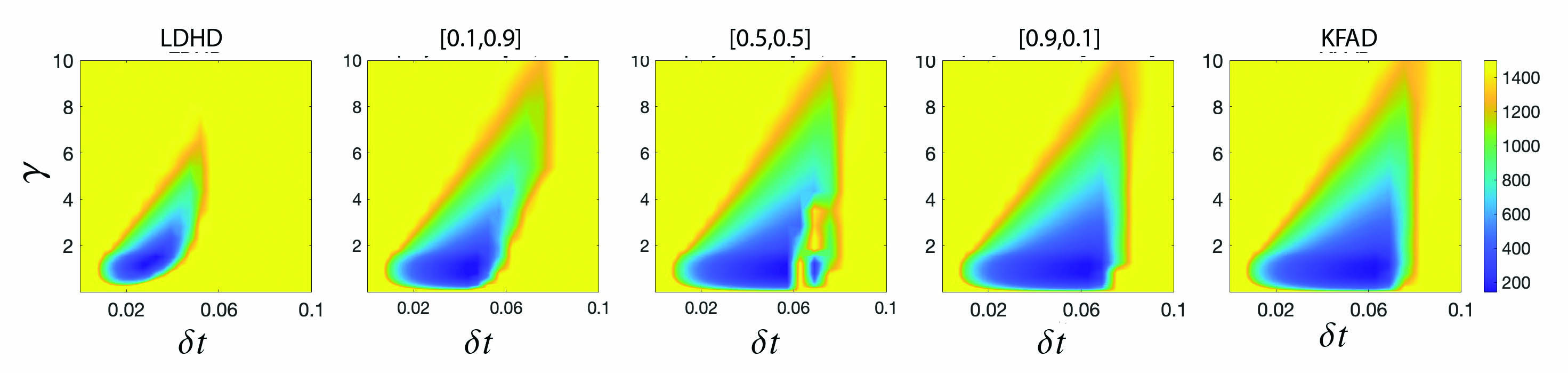}
  \caption{Comparison of number of iterations until convergence as a function of $\gamma$ and $\Dt$, for LDHD versus four different variants of projective McFAD. Each of the plots is averaged over a number of different initializations in the $(x,y)$ domain. The values of $\mu$ and $\alpha$ are fixed for all FAD plots to $\mu = 1$ and $\alpha = 1$. The maximum number of iterations is set to 1500.}
  \label{fig:Rosen_gamma_h_avrg_inits_mu1_a1}
\end{figure}
Looking at Fig.~\ref{fig:Rosen_gamma_h_avrg_inits_mu1_a1}, we see a big difference in the convergence speed of LDHD and the four variants of projective McFAD, the biggest contrast being observed between LDHD and  KFAD. We can see that KFAD allows for fast convergence for a much wider gamut of stepsize and friction. 

To get a clearer picture of what happens in the low $\gamma$ regime, we replot the results in Fig.~\ref{fig:Rosen_gamma_h_avrg_inits_mu1_a1} using a logarithmic scale in the $\gamma$ axis. We also allow the simulations to run for longer (5000 iterations as opposed to 1500 before), to make sure we capture convergence speed differences between the methods at low linear friction, which might not be apparent if we stop the runs too soon.

\begin{figure}[htbp!]
  \centering
  \includegraphics[width=5.5in]{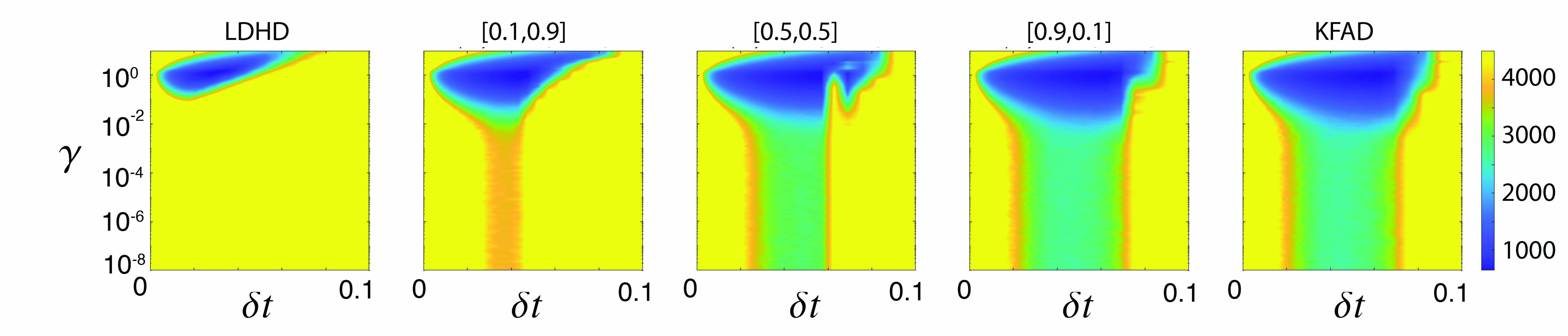}
  \caption{Number of iterations for convergence as a function of~$\gamma$ and~$\Dt$. The pair ``$[\lambda_1, \lambda_2]$'' indicates the parameters used in mixture coupling.  We use a logarithmic scale on the $\gamma$ axis to focus on the low $\gamma$ regime. As before, we use $\mu=1$ and $\alpha=1$.}
  \label{fig:Rosen_gamma_h_avrg_mu_1_alpha_1_logscale}
\end{figure}
As we can see from the results of Fig.~\ref{fig:Rosen_gamma_h_avrg_mu_1_alpha_1_logscale}, KFAD can converge in the very low $\gamma$ regime.  This behavior can be connected to the interpretation of $\xi$ as an ``adaptive" friction coefficient, which can compensate for low $\gamma$ \cite{NH1,NH3,JoLe2011}. LDHD has no such mechanism and therefore reduces, effectively, to Hamiltonian dynamics at low friction.
Finally, we reproduce the $(\gamma,\Dt)$ graphs of Fig.~\ref{fig:Rosen_gamma_h_avrg_inits_mu1_a1} in Fig.~\ref{fig:Rosen_gamma_h_avrg_mu_1_alpha_1_ratio}
where, instead of coloring according to the number of iterations it takes to reach the vicinity of the minimum, we color according to the ratio of the number of iterations it takes each of the four projective mixture coupling variants to converge over the number of iterations it takes LDHD to converge.

\begin{figure}[htbp!]
  \centering
  \includegraphics[width=5.5in]{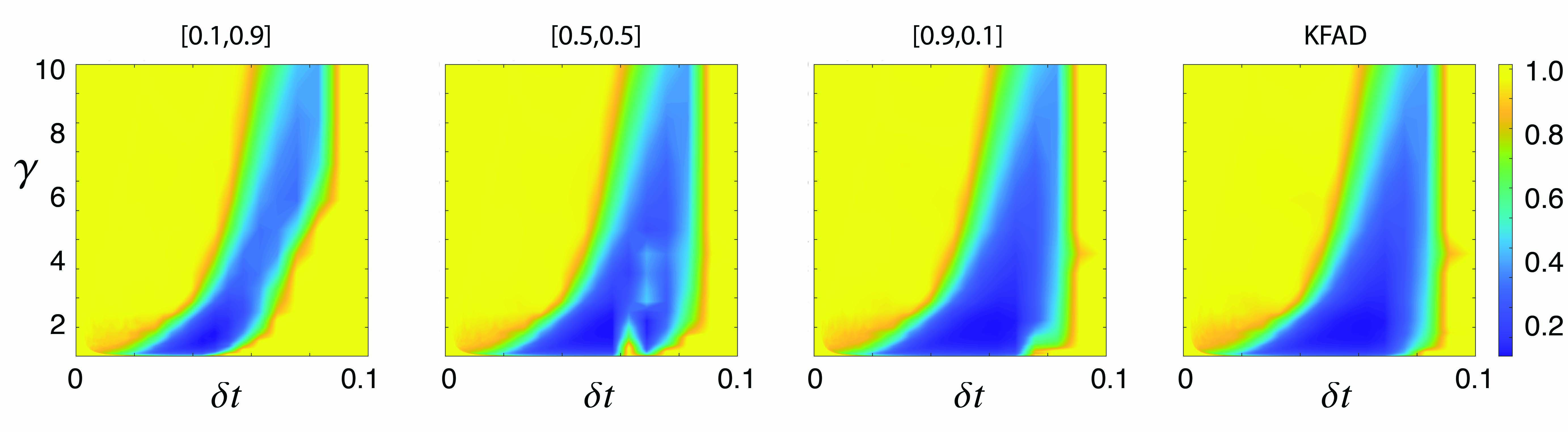}
  \caption{Ratio of number of iterations to reach vicinity of the minimum for projective mixture coupling ($\mu=1, \alpha=1$) compared to LDHD. The pair ``$[\lambda_1, \lambda_2]$'' indicates the parameters used in mixture coupling. }
  \label{fig:Rosen_gamma_h_avrg_mu_1_alpha_1_ratio}
\end{figure}
%
%
Having examined the relationship between~$\gamma$ and~$\Dt$ and how it affects convergence of LDHD and projective mixture coupling, we now proceed to examine the relationship between~$\alpha$ and~$\mu$.  

We choose a set of reasonable values for~$\gamma$ and~$\Dt$, i.e. $\gamma=1 $, $\Dt = 0.05$ and then plot the number of iterations to reach the vicinity of the minimum as a function of~$\mu$ and~$\alpha$ for projective mixture coupling with various parameters. We again average over 1600 different initializations to eliminate the effect of initialization on convergence speed. The results are reported in Fig.~\ref{fig:Rosenbrock_McFAD_mu_alpha_avrg_inits_gamma_1_h_0c05}. 

\begin{figure}[htbp!]
  \centering
  \includegraphics[width=5.5in]{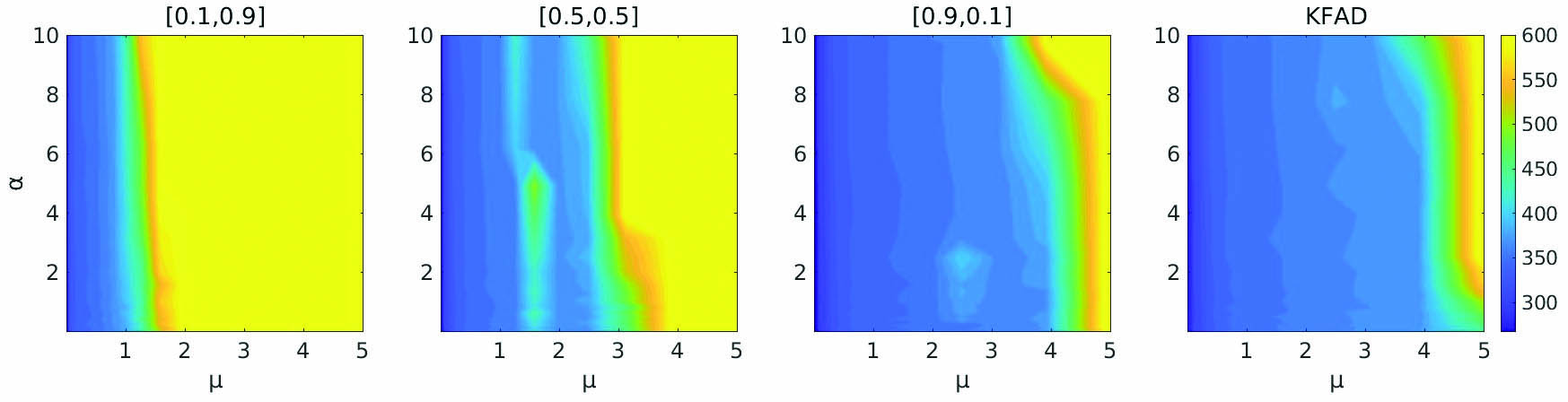}
  \caption{Number of iterations until convergence as a function of $\alpha$ and~$\mu$ for friction $\gamma=1$ and step size~$\Dt=0.05$. The contour plots are colored according to the number of iterations it takes for the four variants of mixture coupling to reach the vicinity of the minimum. The results are averaged over a range of different initializations in the $(x,y)$ domain. }
  \label{fig:Rosenbrock_McFAD_mu_alpha_avrg_inits_gamma_1_h_0c05}
\end{figure}
As we can see in Fig.~\ref{fig:Rosenbrock_McFAD_mu_alpha_avrg_inits_gamma_1_h_0c05}, there is a front beyond which performance deteriorates, but there is a wide range of $(\mu,\alpha)$ pairs on the left of the front that we can choose, resulting in fast convergence. We therefore observe that the algorithm is robust regarding the choice of $\mu$ and $\alpha$. The reason performance deteriorates for high~$\mu$ is because the limit of friction-adaptive descent as~$\mu \to \infty$ is just LDHD.  On the other hand, at fixed $\delta t$, and for small $\mu$, the tight coupling between kinetic energy and $\xi$ will introduce instability in the numerical method.  Thus we recommend  values of $\mu$ that are bounded in a moderate interval around unity. This is reminiscent of a similar recommendation for Adaptive Langevin dynamics~\cite{LeSaSt2020}.

Our primary objective so far has been to demonstrate that FAD can lead to much faster convergence compared than LDHD. We are not necessarily aiming to show that FAD outperforms other popular optimization methods. It is however be interesting to see how KFAD compares to Adam \cite{kingma2014adam}. We formulate Adam as a system of ODEs, similarly to \cite{da2020general}. The continuous equations we use for Adam are
\begin{align}
    \dot{x} &= \frac{p}{\sqrt{\zeta}+\epsilon}, \label{eq:adam_x} \\
    \dot{p} &= -\nabla f(x) -\gamma p,  \label{eq:adam_p} \\
    \dot{\zeta} &=  \nabla f(x)^2 - \alpha \zeta,  \label{eq:adam_zeta},
\end{align}
where $\zeta \in \mathbb{R}^N$ and $\nabla f(x)^2$ denotes element-wise squares.
We use the following splitting method to discretize the dynamics~\eqref{eq:adam_x}-\eqref{eq:adam_zeta}:
\begin{equation}\label{eq:Adam_splitting}
\begin{pmatrix}\dot{x} \\ \dot{p} \\ \dot{\zeta} \end{pmatrix}
 =\underbrace{\begin{pmatrix} \frac{p}{\sqrt{\zeta}+\epsilon}\\0\\0 \end{pmatrix}}_{\text{A}} + 
  \underbrace{\begin{pmatrix}0\\-\nabla f(x)\\0 \end{pmatrix}}_{\text{B}} +
  \underbrace{\begin{pmatrix}0\\0 \\  \nabla f(x)^2 - \alpha \zeta \end{pmatrix}}_{\text{C}}+
  \underbrace{\begin{pmatrix}0\\- \gamma p\\ 0 \end{pmatrix}}_{\text{D}} .
\end{equation}

Parts B and D are discretized as explained in section \ref{sec:numerical_schemes}. Step A can be solved exactly as $x_{n+1} = x_n + \delta t \frac{p_n}{\sqrt{\zeta_n}+\epsilon}$. For step C we have $\zeta_{n+1} = \mathrm{e}^{-\alpha \delta t} \zeta_n + \frac{F^2}{\alpha} (1-\mathrm{e}^{-\alpha \delta t})$.

We choose $\alpha=1$ for Adam and $\alpha=1$, $\mu=1$ for KFAD. We set the maximum number of iterations to 2500 and average over 1600 different initializations. 
The plots in Figs.~\ref{fig:adam_ldhd_kfad_gamma_h_Rosen_2d} and~\ref{fig:adam_ldhd_kfad_gamma_h_Rosen_2d_log} are colored according to the number of iterations it takes for the optimizer to reach the vicinity of the minimum. The horizontal axis corresponds to the step size $\delta t$ and the vertical axis to the friction $\gamma$. Fig.~\ref{fig:adam_ldhd_kfad_gamma_h_Rosen_2d_log} is the same as Fig.~\ref{fig:adam_ldhd_kfad_gamma_h_Rosen_2d}, but with a logarithmic scale on the $\gamma$ axis. This allows us to see that for the 2D Rosenbrock function, KFAD is more efficient than both LDHD and Adam in the low $\gamma$ regime. Adam on the other hand demonstrates faster convergence for the high $\gamma$/low step size regime.

\begin{figure}[htbp!]
    \centering
    \includegraphics[scale=0.27]{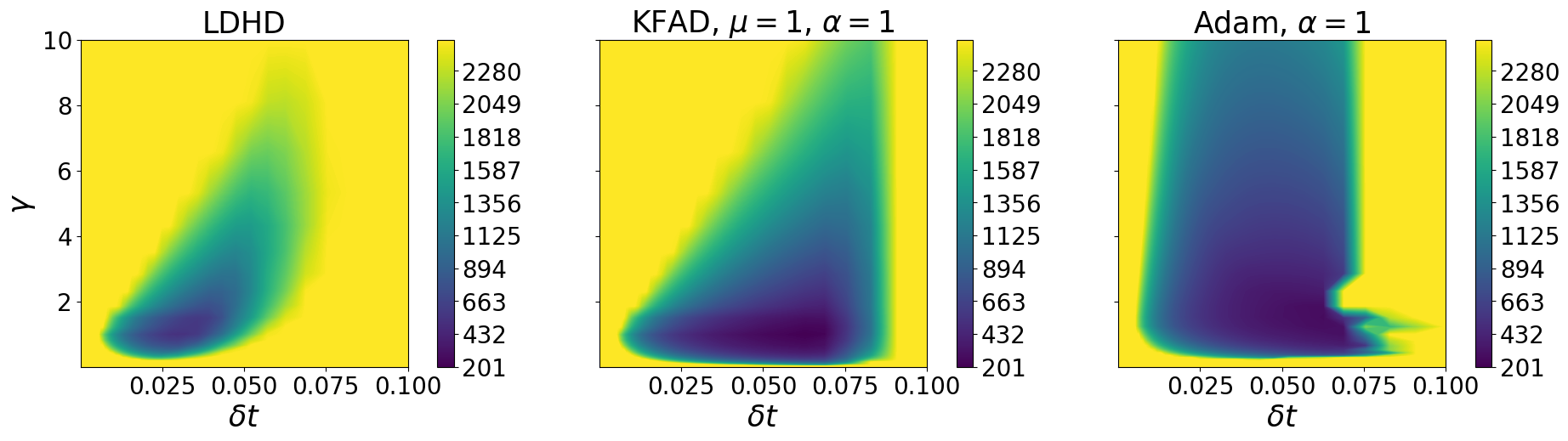}
    \caption{Number of iterations for LDHD, KFAD and Adam to reach the vicinity of the minimum of the 2D Rosenbrock function as a function of $\gamma$ and $\delta t$. Results are averaged over 1600 different initializations.}
    \label{fig:adam_ldhd_kfad_gamma_h_Rosen_2d}
\end{figure}
\begin{figure}[htbp!]
    \centering
    \includegraphics[scale=0.27]{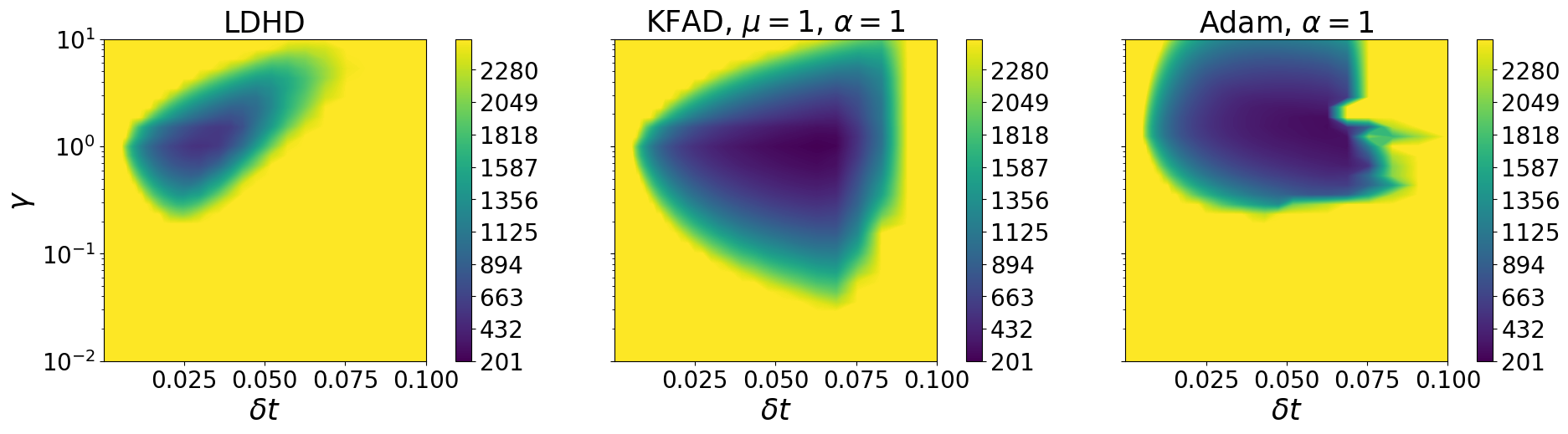}
    \caption{Number of iterations for LDHD, KFAD and Adam to reach the vicinity of the minimum function as a function of $\gamma$ and step size $\delta t$. Results are averaged over 1600 different initializations. We use a logarithmic scale on the $\gamma$ axis for better visualization of the low-$\gamma$ regime. We can see that KFAD is faster than both LDHD and Adam in the low-$\gamma$ region.}
    \label{fig:adam_ldhd_kfad_gamma_h_Rosen_2d_log}
\end{figure}

To make visual comparison of the three methods easier, we show the ratio of iterations it takes for KFAD to converge over those of Adam in Fig.~\ref{fig:adam_kfad_gamma_h_Rosen_2d_ratios}. We see that there is a green and yellow region, where KFAD is between two to four times slower than Adam. There are also two regions, a horizontal band for low $\gamma$'s and a vertical band for higher step sizes  where KFAD is between three and ten times faster than Adam. We can see that KFAD allows the use of higher step sizes and allows for fast convergence in the high high step size regime for a wide range of $\gamma$'s.
\begin{figure}[htbp!]
    \centering
    \includegraphics[scale=0.29]{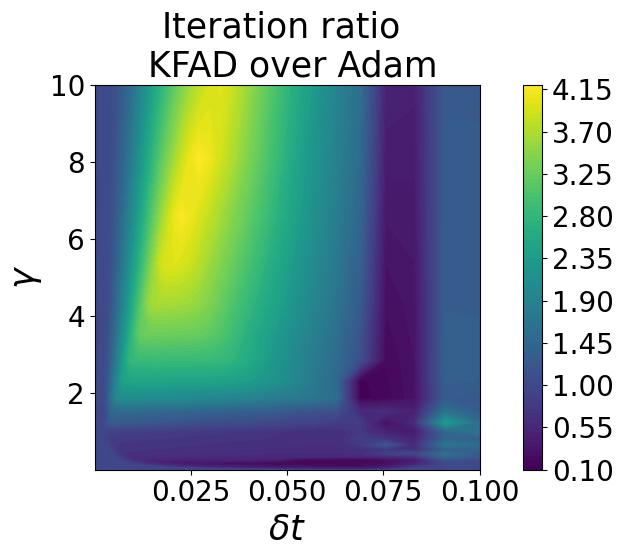}
    \caption{Ratio of number of iterations it takes to reach the vicinity of the minimum for KFAD ($\mu=1, \alpha = 1$) compared to Adam ($\alpha = 1$).}
    \label{fig:adam_kfad_gamma_h_Rosen_2d_ratios}
\end{figure}

We would like to stress again that it was not our aim to necessarily compete with Adam, rather we wanted to show that FAD methods are competitive with momentum SGD and to demonstrate how cubic damping can be used as an effective mechanism to suppress oscillations and speed up convergence. In fact, it is possible to blend the ideas underpinning Adam with the ones of FAD in order to hopefully get the best features of both algorithms. We are currently working on this.
%
\subsection{Lennard--Jones and Morse Clusters}
\label{sec:num_clusters}
As an interesting and nontrivial test problem we have considered the optimization of atomic clusters consisting of $N$ atoms moving in $\R^3$ and interacting in a pairwise (distance-based) potential.  The total potential energy is
\[
U(q) = \sum_{i=1}^{N-1} \sum_{j=i+1}^N \varphi(\|q_i-q_j\|),
\]
where $q_i$ represents the three component vector of position coordinates of the $i$th atom.  This problem has been extensively studied as a test case for global optimization and many specialized methods have been developed \cite{wales}. A significant challenge is the high dimensionality of these problems and the proliferation of local minima, in some cases nearly at the same depth as the global one.  As our first test case we used a system of 38 atoms in the Lennard--Jones pair potential defined by 
\[
\varphi_{\rm LJ}(r) = 4 \left( r^{-12} - r^{-6} \right).
\]
We consider here the optimization for initial conditions taken in the vicinity of the (unique) global minimum which is unique up to rotations and translations.

In this exploration, we attempted again to use FFAD and various mixture coupling methods, but found them, as for the Rosenbrock function, to be of little benefit in terms of efficiency. It should be clarified here that while FFAD was effective in damping oscillations, as was demonstrated for example in Figs.~\ref{FFADvKFAD} and~\ref{McFAD}, the very mechanism that facilitated the damping of oscillatory modes seemed to have an adverse effect on convergence speed, which is what we mean when referring to efficiency (see the discussion after~\eqref{eq:F_eff}). Therefore our focus here is on KFAD and its performance relative to LDHD.

To generate an initial condition for the optimizer, we proceeded as follows.  We started at the known (tabulated) minimum energy configuration.  We then applied stochastic dynamics (Langevin dynamics) for a short period (400 steps, $\delta t=0.001$, $\beta^{-1} =2$ with $\beta$ proportional to the inverse temperature, similarly to~\eqref{nh-3}) producing a new state with a significantly higher energy.  The challenge then for optimization is to return to the ground state.   An example starting structure can be seen in  the left panel of Fig.~\ref{fig:LJ38clusters}. It can be compared with the right panel which shows the highly ordered minimum energy cluster.

\begin{figure}[htbp!]
  \centering
  \includegraphics[width = 2in]{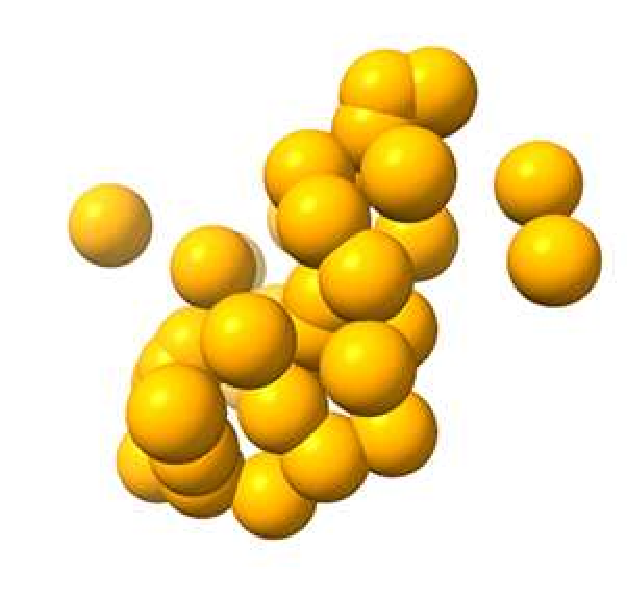}\hspace{0.3in} \includegraphics[width = 2in]{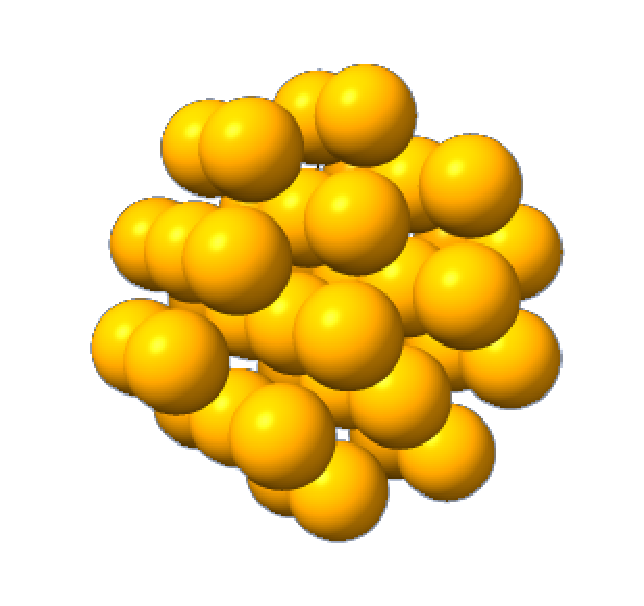}
  \caption{Instances of 38 atom clusters for Lennard--Jones interaction potential.  Left: starting state for optimization tests, energy = $-141.77$. Right: global minimizer, energy = $-173.91$.}
  \label{fig:LJ38clusters}
\end{figure}

\subsubsection{Optimal choice of $\alpha$ and $\mu$}
We tuned the artificial parameters $\alpha$ and $\mu$ and the friction $\gamma$ of KFAD by considering various choices of all three.  We found that KFAD performed best with small values of $\gamma$ and with moderate values of $\mu$ and $\alpha$.  It was somewhat sensitive to the selection of these two parameters but less so than LDHD was to the choice of the linear damping~$\gamma$. We show one series in Fig.~\ref{fig:KFAD_alpha_mup01} which illustrates the convergence at $\mu=0.1$ for various values of~$\alpha$.

\begin{figure}[htbp!]
  \centering
  \includegraphics[width = 5in]{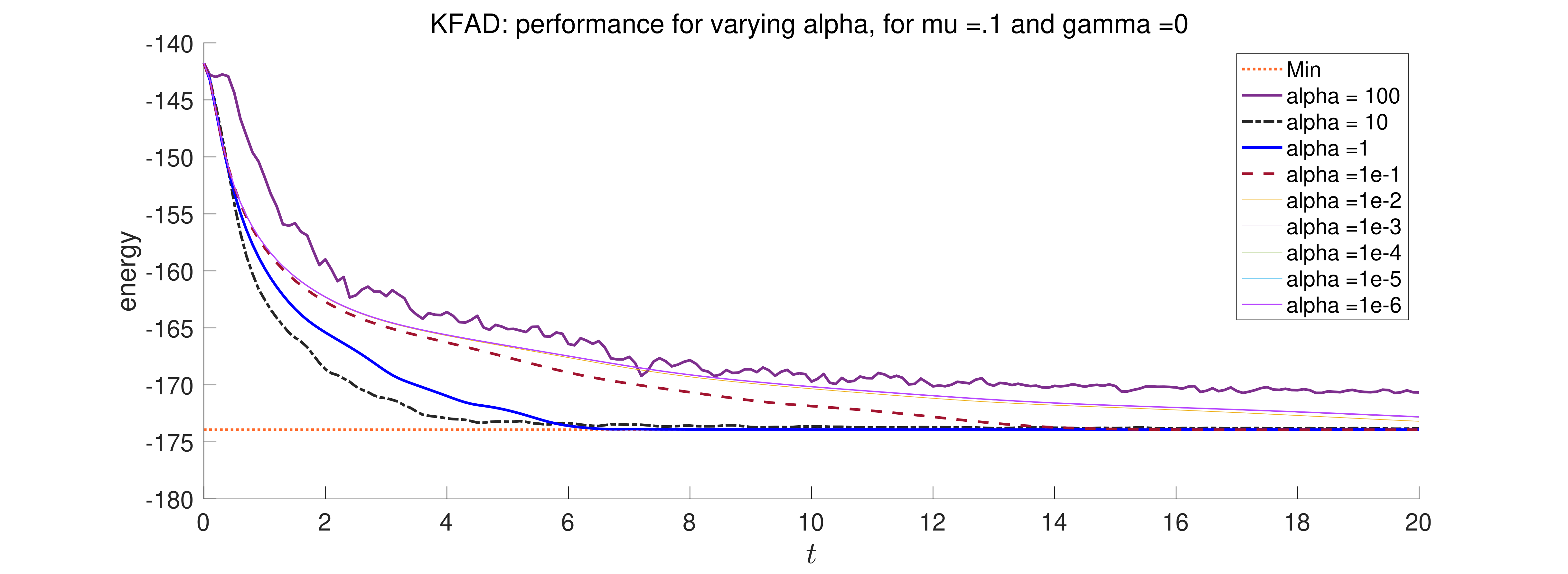}
  \caption{Lennard--Jones 38: exploration of the optimal choice of $\alpha$ for $\mu =0.1$. The objective function is reported as a function of the physical time.}
  \label{fig:KFAD_alpha_mup01}
\end{figure}

When $\mu=0.1$, convergence is seen for all values of $\alpha<100$ and is particularly rapid for $\alpha =1$ or $\alpha =10$.   

\subsubsection{Exploration of LJ75}
Having investigated the parameters, we considered direct comparisons between KFAD and LDHD on a specific challenge related to the 75-atom system. The minimum for $N=75$ is non-icosahedral, with energy~$-397.492331$ (see~\cite{DoWaBe1995}).  
We again followed the standard practice of initializing molecular optimization from Langevin samples.  From the minimizer we allowed 400 steps of Langevin dynamics (BAOAB scheme~\cite{LeMa2013}, $\beta=0.5$, stepsize $\Dt=0.001$ with the friction coefficient in Langevin set to~$\gamma=20$).   The result of this process was an initialization at an energy between approximately~$-260$ and~$-230$, significantly higher than the global minimum.

We found that the previous parameter testing in the case of LJ38 was transferable to the larger molecule.
We fixed the parameters of KFAD at $\mu=0.1$, $\alpha=10$, $\gamma = 10^{-5}$ and considered various choices of friction in LDHD.  We found that LDHD required some tuning of the friction and specifically that friction needed to be greater than~1.  We allowed 2000 steps for optimization by each scheme.

In our experiments we observed that the performance of KFAD was relatively independent of the momentum initialization.  In our first set of experiments, we initialized both KFAD and LDHD using the thermalized momenta corresponding to $\beta = 0.5$ which were obtained during the initial equilibration phase.  This is certainly not an uncommon initialization strategy, but we found that for LDHD the results obtained using these momenta were very poor.  In most cases the resulting runs did not produce approximations to the minimum.

We show in Figs.~\ref{fig:minima} and~\ref{fig:local} a comparison of the energies achieved. KFAD reached the global minimum in more than 83\% of runs.  For no value of linear damping did LDHD find the minimizer in more than 15\%  of trials; in most cases the performance was much worse.  

\begin{figure}[htbp!]
  \centering
  \includegraphics[width = 5in]{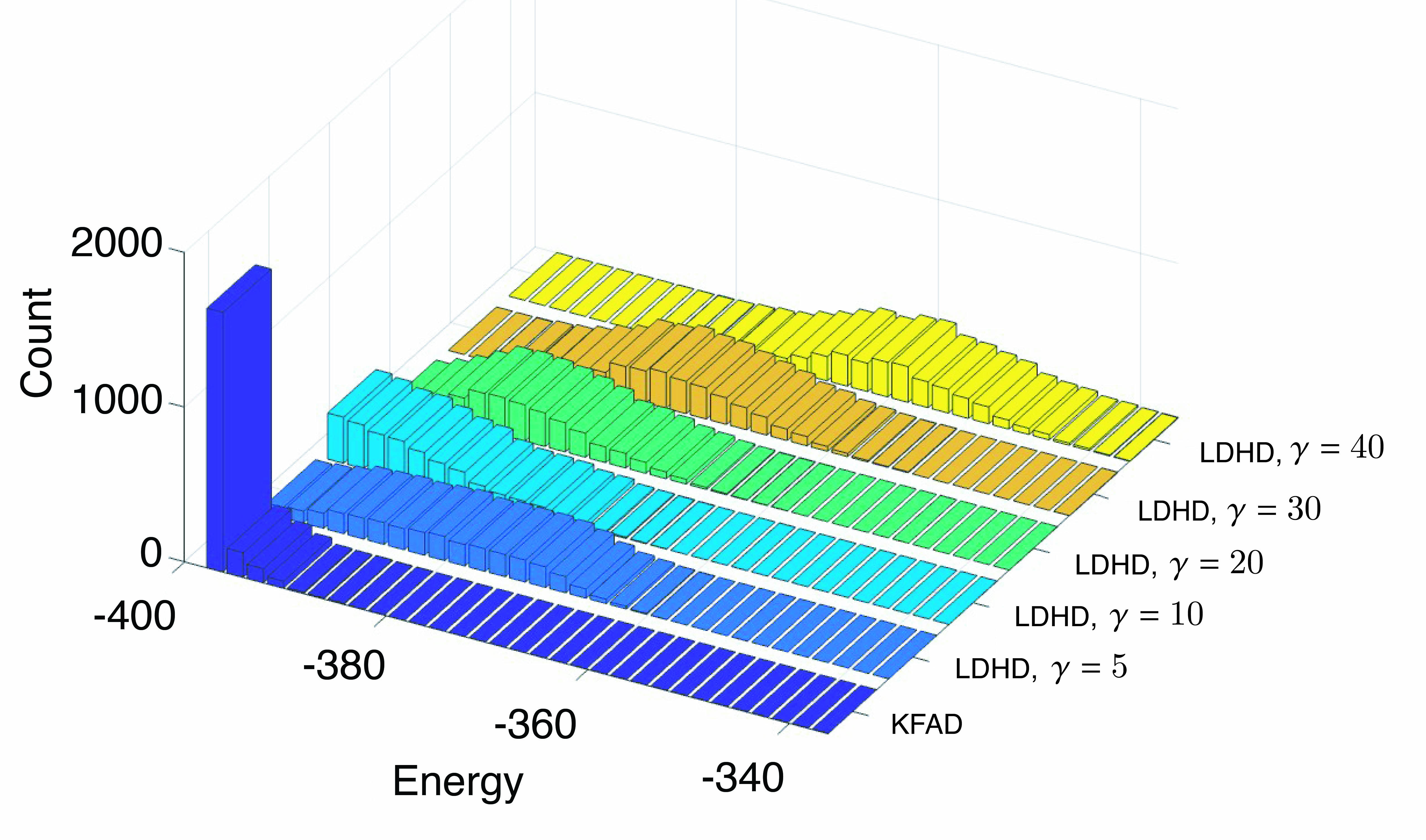}
  \caption{Minimizers for runs of the LJ-75 system initiated from the canonical ensemble.  The comparison shows the states obtained by LDHD (with various linear damping coefficients) and KFAD, arranged by energy level. }
  \label{fig:minima}
\end{figure}

\begin{figure}[htbp!]
  \centering
  \includegraphics[width = 4in]{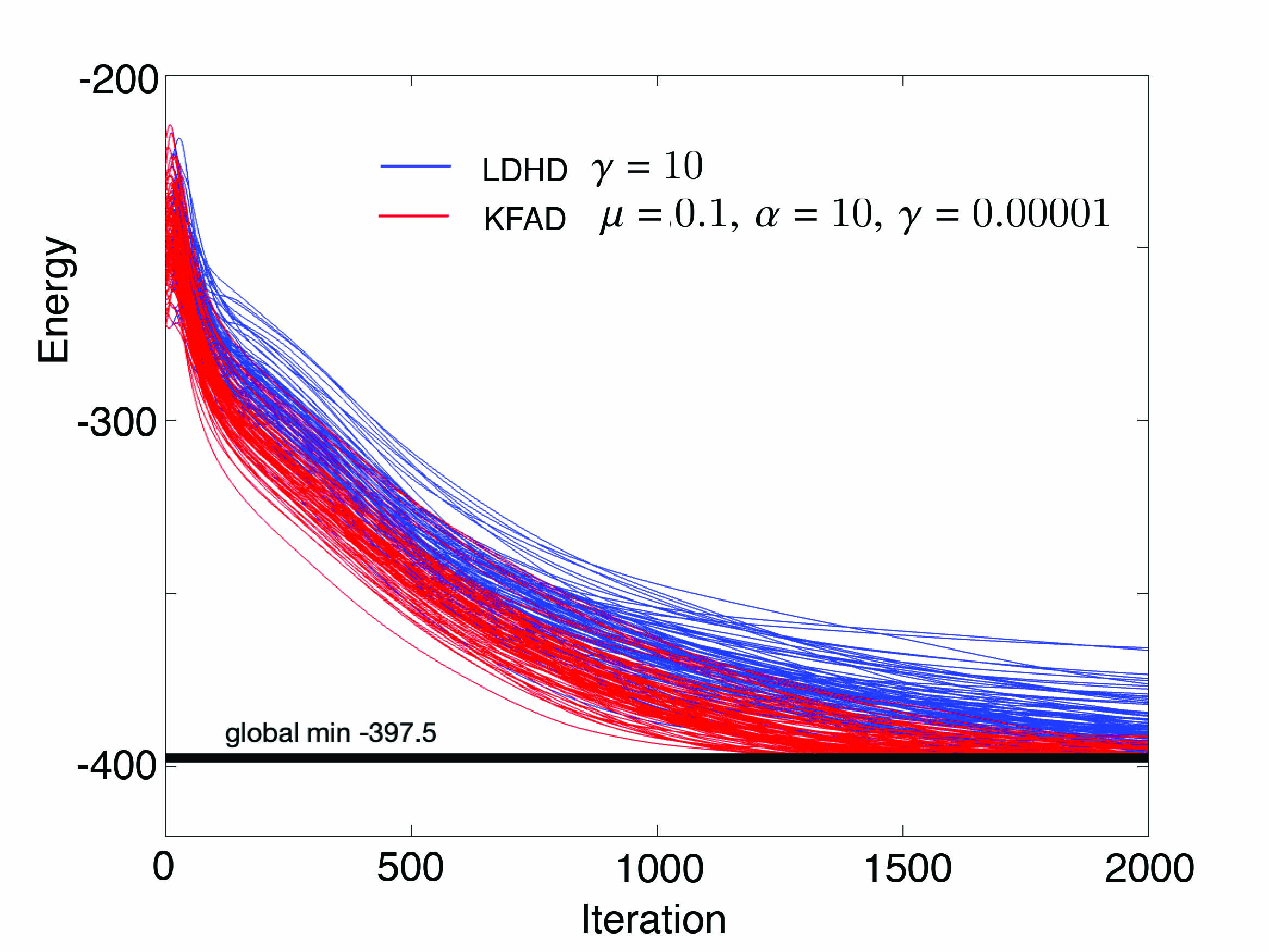}
  \caption{Optimization paths show the decay of energy for LJ-75 in a typical case with momenta initialized from the canonical ensemble.  We see that LDHD is typically significantly slower in descent and ends at higher values than in the KFAD runs.}
  \label{fig:local}
\end{figure}

These results suggest that KFAD is capable of controlling the descent process even in the setting that it starts with a high kinetic energy, which is perhaps not surprising given that it has an in-built kinetic energy control.   We then initialized the system in a different way. After producing a canonical sample using Langevin dynamics as above, we simply set the momenta to~0.  In this setting, we rely on properties of the potential surface to
guide the state of the system toward the minimum.   When started in this way, the results obtained by LDHD were strikingly different; see Fig.~\ref{fig:minima2}. With friction properly tuned, LDHD found the minimum rapidly (as rapidly as KFAD) and in 93\% of cases, whereas KFAD only succeeded in around 77\% of runs.  This suggests that standard dissipation mechanisms are able to find nearly optimal paths in this example.  Increasing the distance to the minimum (by longer
equilibration) typically resulted in both methods getting stuck in an elevated energy state. 

\begin{figure}[htbp!]
  \centering
  \includegraphics[width = 5in]{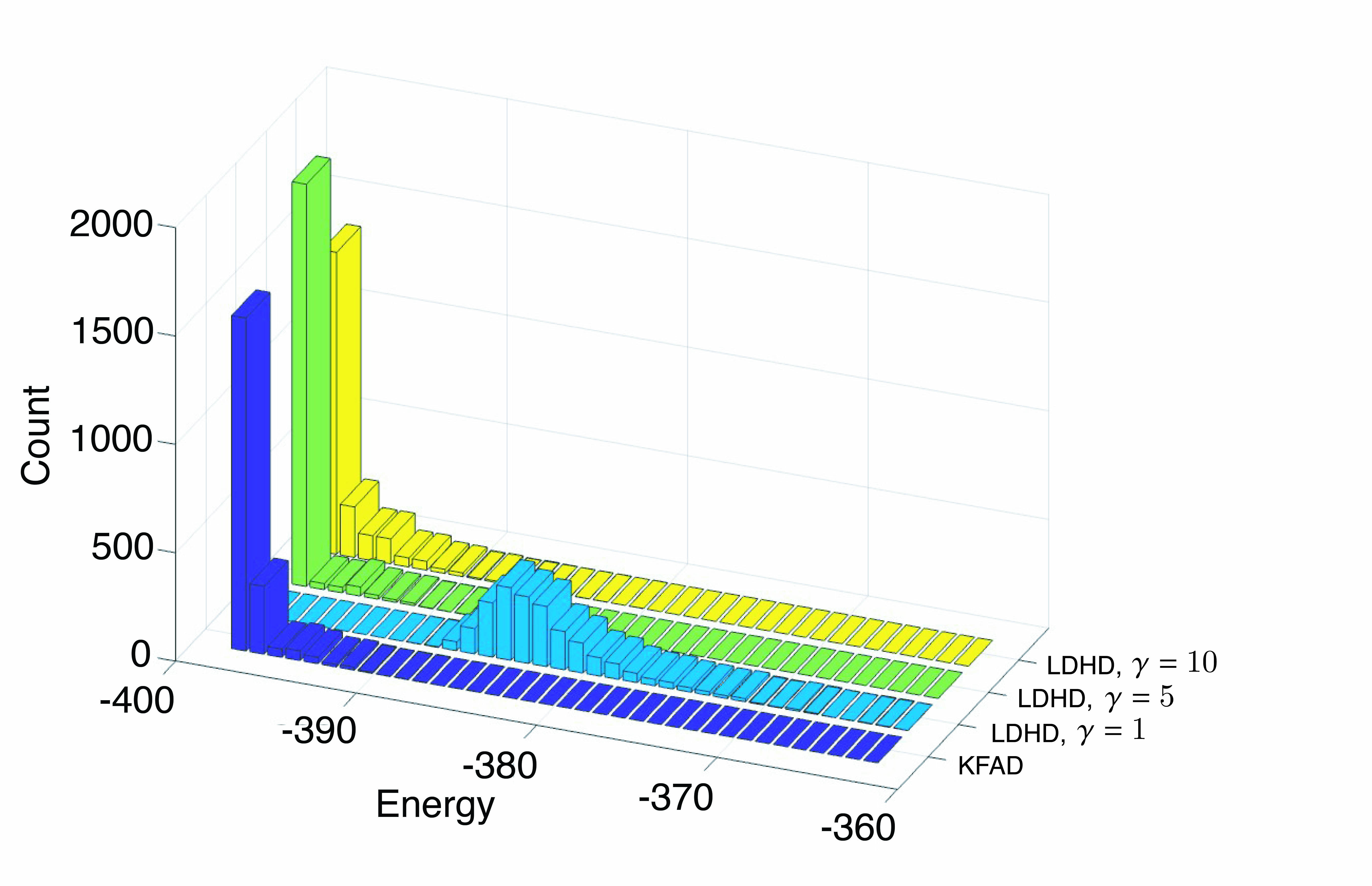}
  \caption{Minimizers for runs of the LJ-75 system initiated by zeroing the momenta at the start. }
  \label{fig:minima2}
\end{figure}

\subsubsection{Approximate global minimization of a Morse cluster}
Although the results for LJ clusters initialized near the global minimum were  inconclusive, 
we attribute this to the steepness of the Lennard--Jones pair potential and the consequent strong propensity for trapping into a single, nearly-harmonic basin.  We next turned our attention to a Morse cluster in which the pair potential is defined by
\[
\varphi_{\rm Morse}(r) = \left(1-\rme^{-a(r-r_0)}\right)^2.
\]
Such clusters are discussed in depth in several articles in the chemistry literature \cite{DoWaBe1995,WaDo1996,doye1996effect,doye1997structural, miller1999structural}. Morse potentials were also extensively used in \cite{wales2001microscopic} to study why long-range potentials
give rise to simpler landscapes.  Using the Morse potential in place of the Lennard-Jones potential undoubtedly changes the problem, but the Morse cluster is still regarded as a difficult optimization problem.

For modest values of~$a$, the propensity for trapping is reduced compared to Lennard--Jones and it is possible to consider the use of FAD methods for
(approximate) global optimization.  We used $a=3$ and $r_0=1$ which is one of the cases studied by Doye, Wales and Berry~\cite{DoWaBe1995}.    We initialized a system of 64 atoms by simply placing the atoms on the vertices of a Cartesian lattice $(x_i,y_i,z_i)$ of the form $\{0,1,2,3\} \times \{0,1,2,3\} \times \{0,1,2,3\}$.  This leads to a rather large initial energy.  We then added an equilibration phase using Langevin dynamics to randomize the starting point (1000 steps, stepsize $\delta t=0.001$, $\beta=0.1$).  After this we applied various optimization methods.  The initialization procedure results in a very high starting energy in~$[-50,0]$, compared to the global optimum energy of~$-512.83$. In this setting we found it did not matter how we initialized momenta, so we set the momenta to zero for simplicity.

We explored the choice of stepsize.  Choosing the right stepsize is nontrivial. There is a significant range below the stability threshold where the LDHD method results in a poor accuracy, regardless of the friction used.  Even though LDHD appeared to be stable (in the sense of bounded energy) for $\delta t <0.1$ the results were much better with LDHD if $\delta t<0.05$, whereas KFAD was less sensitive to the choice of~$\Dt$.  This probably is due to 
the LDHD paths ``popping out'' of some basins or missing some relevant pathways during minimization.   The fact that we can, at least in some cases, use large steps and still obtain accurate minima using KFAD is a favorable feature.    In LDHD,  the optimal friction was found to lie around $\gamma =1$; at larger stepsize, higher friction (e.g. $\gamma \approx 2$) gave better results, suggesting that the friction was compensating for numerical error.  We found that $\alpha=\mu=1$ gave the best results in KFAD, although the choice was not critical, and in KFAD, the optimal linear damping was $\gamma=0$, so all damping comes from the use of the auxiliary term.  The use of zero friction in the Morse cluster example stands in stark contrast to our observation that a moderate friction is needed in KFAD when it is applied to the Rosenbrock function.

Unlike in the case of the local minimization of Lennard--Jones we considered earlier, in this optimization problem we need relatively long trajectories.  We initially used 50,000 steps in each run.  The results shown in Fig.~\ref{fig:morse} clearly demonstrate that in this case the KFAD method can reach an energy near to that of the global minimum, with high reliability, whereas the LDHD method fails in most cases.  We did not pursue the challenge of specifically finding the global minimum, which would require additional machinery, since our interest here was simply the comparison of native LDHD to KFAD.

\begin{figure}[htbp!]
  \centering
  \includegraphics[width = 5.8in]{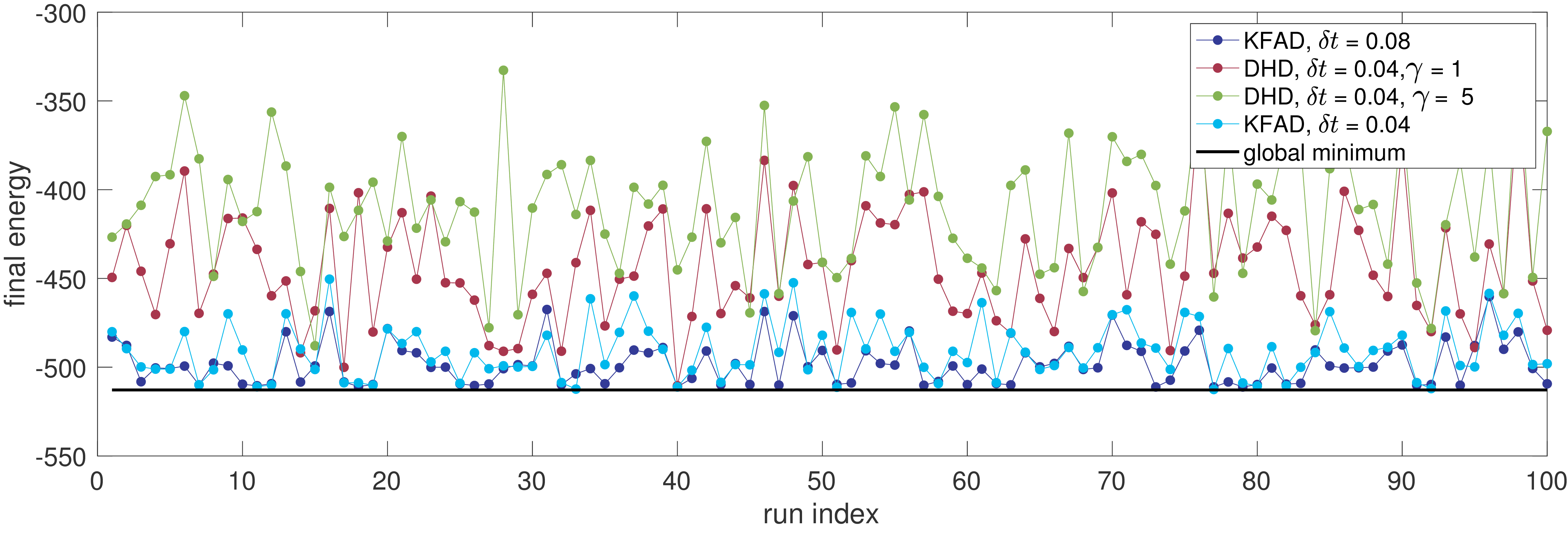}
  \caption{Energies obtained using LDHD and KFAD for the 64 atom Morse cluster.  The KFAD method approaches the minimum energy in almost all runs, whereas LDHD, for all choices of friction tested, does not come close. Results are shown for KFAD at stepsizes of both $\delta t=0.04$ and $\delta t=0.08$, but LDHD only for the smaller stepsize since its performance was much worse at larger stepsize.  Moreover, in KFAD, the larger stepsize gave visible improvements over the smaller one, suggesting that exploration was enhanced.}
  \label{fig:morse}
\end{figure}

We next conducted several series of 100 runs with each of three stepsizes and three choices of $\gamma$ (in LDHD); in this case taking~$10^5$ steps each time.  We have graphed all the obtained minimum energies in each series in Fig.~\ref{fig:morse_pars100K}, along with the mean of all the runs of the series.  
\begin{figure}[htbp!]
  \centering
  \includegraphics[width = 5in]{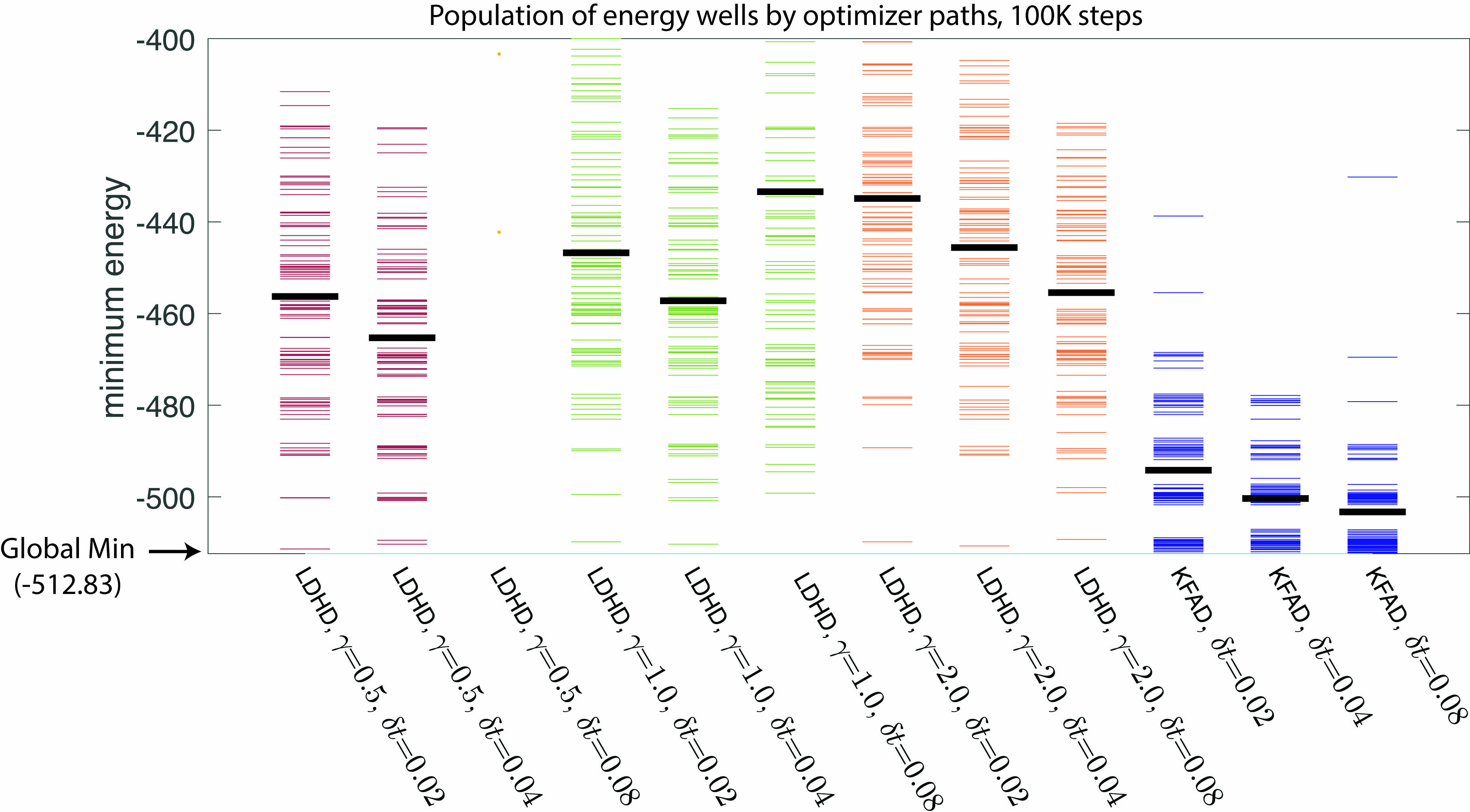}
  \caption{A diagram comparing minimum energies obtained for the 64 atom Morse cluster in long runs, for different choices of stepsize and, for LDHD, for different choices of friction. Parameters for KFAD were $\alpha=\mu=1$ and $\gamma=0$.  The solid black bars show the means of all the minima of the column.  The graph shows that the minima achieved by KFAD are overall much lower than those achieved by LDHD for various parameters.  We also see that KFAD obtains improved accuracy when using a large stepsize ($\delta t=0.08$) for which LDHD would be inaccurate.}
  \label{fig:morse_pars100K}
\end{figure}

In this experiment, we see that KFAD trajectories appear to  accumulate in lowest energy basins, whereas LDHD paths appear to get stuck at elevated energies.  To test this hypothesis we reran the previous set of experiments using~$10^6$ steps in every run to see how the figure changed.  These results are reported in Fig.~\ref{fig:morse_pars1M}.   In this example, it appears that KFAD paths are finding transitions between states that are rarely found by dissipated Hamiltonian descent.  At the largest stepsize, most of the minimizers are at (or within 1\% of) the global minimum.

\begin{figure}[htbp!]
  \centering
  \includegraphics[width = 5in]{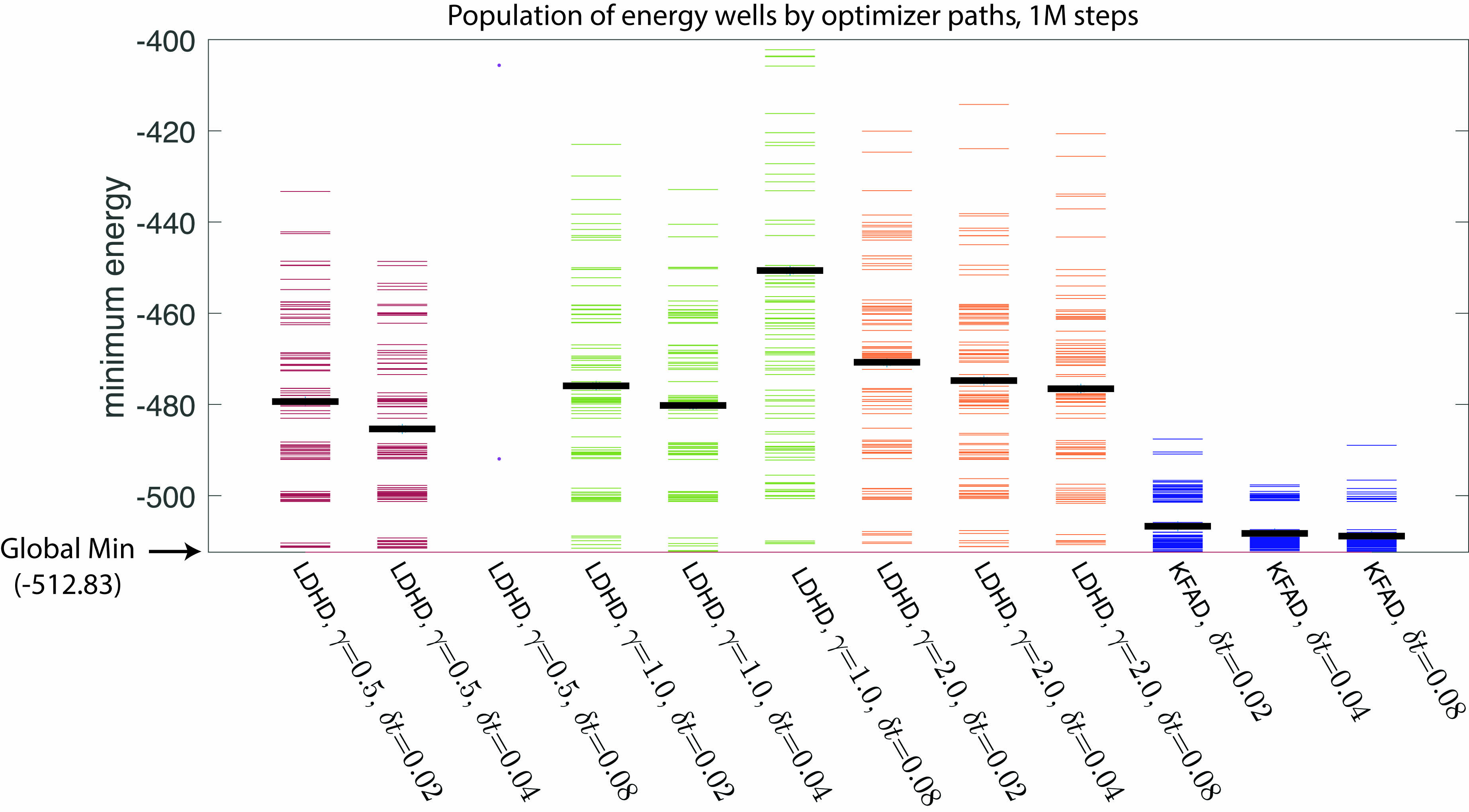}
  \caption{The energies of the 64 atom Morse cluster obtained in runs of~$10^6$ steps (compare with Fig.~\ref{fig:morse_pars100K}).  We call attention to the fact that for $h=0.08$ in KFAD (last column), nearly all the trajectories that had previously landed in the well at depth around~$-500$ have moved into the minimum energy state (or very near of it).}
  \label{fig:morse_pars1M}
\end{figure}

A  plausible explanation for the substantial improvement of KFAD over LDHD in this case is the following.  Compared to LDHD, the KFAD method, which is used here without linear damping (i.e. $\gamma=0$), has very mild damping in the later stages of minimization (when~$\|p\|$ is small).  Compared to LDHD,  KFAD looks more like Hamiltonian dynamics at low kinetic energy.   This increases the ability of the method to continue to explore the low-lying minima of the potential landscape.  We plan to further study this issue in application settings in subsequent work.

\section{Conclusions}
%
We have considered an alternative to dissipated Hamiltonian descent as a scheme for local optimization of smooth potential functions based on Nos\'{e}--Hoover dynamics and its generalizations; these methods embed Hamiltonian dynamics in a framework including one or more auxiliary variables.  We have only explored some of the simplest methods that can be built on this foundation and have shown convergence, both for the dynamics itself, and for simple discretizations of the dynamics.  

We have seen that although the FAD methods are related to nonlinear (cubic) damping, we did not work in that parameter regime.  The tuning we have done of $\alpha$ and $\mu$ in our numerical studies showed that optimal values of the parameters were moderate, typically of order~1. Thus it appears that the flexible relaxation of $\xi$ is an important part of the performance enhancement we obtained.

In terms of practical application, it seems the KFAD scheme, which is particularly simple to implement and which introduces virtually no computational overhead, is already of great potential utility.  The experiments on Rosenbrock systems hint at the robustness of these methods.  The efficiency of the methods was also demonstrated for more challenging examples, including a 192 degree-of-freedom Morse cluster, with highly anharmonic potential landscape, where the KFAD method was able to reach near-global energy minimizing states.  For this global optimization problem there may of course be other schemes besides LDHD which would make more realistic comparisons, but we were impressed that a pure dynamics-based method could find global optima without the trapping clearly visible in the LDHD runs (note in particular the results for the Morse cluster with~$10^6$ steps and $\delta t =0.08$, where 86\% of the miminizers obtained have energies within 1\% of the global minimum.).

Although the FFAD method showed a strong benefit in terms of effectiveness in suppressing oscillations in initial experiments with the Rosenbrock problem, it generally performed worse than KFAD. This is due to the nature of the FFAD dynamics, which as we explain does reduce oscillations, but also tends to slow down descent towards the minimum. We suspect that the idea of adding some directed cubic damping  will nonetheless prove useful in some other settings.

We have barely scratched the surface of this method-family and look forward to gaining further insight into the range of application for FAD methods in future studies.  Of particular interest are activated (driven) systems in which the optimization is continually disturbed by time-dependent forcing and systems with stochastic forcing (e.g stochastic gradients, as arise in deep learning).  We are currently exploring these models in follow-on studies.

\label{sec:conclusions}



\section*{Acknowledgments}
The work of AK and BL was funded by the Engineering and Physical Sciences Research Council Grant EP/S023291/1 (MAC-MIGS Centre for Doctoral Training) and EP/R014604/1 (Isaac Newton Institute), specifically funding provided by its Data-Driven Engineering programme (Spring 2023). The work of AK was also supported by IBM Research. The work of GS was funded by the European Research Council (ERC) under the European Union's Horizon 2020 research and innovation programme (project EMC2, grant agreement No 810367), and by Agence Nationale de la Recherche, under grants ANR-19-CE40-0010-01 (QuAMProcs) and ANR-21-CE40-0006 (SINEQ).  BL acknowledges helpful discussions with D. Wales regarding the atomic clusters of Sec. 6.
\bibliographystyle{amsplain}
\bibliography{references}

\end{document}